\documentclass[psamsfonts]{amsart}

\usepackage{amsmath,amssymb,amsthm,amsfonts,amsbsy,latexsym,dsfont,color,enumitem, upgreek, xcolor, xfrac, times}
\usepackage[numeric,initials,nobysame]{amsrefs} 
\usepackage[T1]{fontenc}
\usepackage[mathscr]{eucal}
\usepackage[unicode,bookmarks,colorlinks]{hyperref}
\usepackage{graphicx}

\hypersetup{ linkcolor=brickred }
\definecolor{mahogany}{cmyk}{0, 0.77, 0.87, 0}
\definecolor{salmon}{cmyk}{0, 0.53, 0.38, 0}
\definecolor{melon}{cmyk}{0, 0.46, 0.50, 0}
\definecolor{yellowgreen}{cmyk}{0.44, 0, 0.74, 0}
\definecolor{brickred}{cmyk}{0, 0.89, 0.94, 0.28}
\definecolor{OliveGreen}{cmyk}{0.64, 0, 0.95, 0.40}
\definecolor{RawSienna}{cmyk}{0, 0.72, 1.0, 0.45}
\definecolor{ZurichRed}{rgb}{1, 0, 0}

\newtheorem{theorem}{Theorem}[section]
\newtheorem{corollary}[theorem]{Corollary}
\newtheorem{proposition}[theorem]{Proposition}
\newtheorem{lemma}[theorem]{Lemma}

\newtheorem{remark}[theorem]{Remark}

\newtheorem{question}[theorem]{Question}
\newtheorem{conjecture}[theorem]{Conjecture}

\theoremstyle{remark}

\numberwithin{equation}{section}

\newcommand{\ext}{\mathrm{ext}}
\newcommand{\dis}{\mathrm{dis}}

\newcommand{\R}{\mathbb{R}}
\newcommand{\Z}{\mathbb{Z}}
\renewcommand{\P}{\mathbb{P}}

\newcommand{\bC}{\mathbb{C}}

\newcommand{\bE}{\mathbb{E}}

\newcommand{\bH}{\mathbb{H}}

\newcommand{\bR}{\mathbb{R}}

\newcommand{\cB}{\mathcal{B}}

\newcommand{\cF}{\mathcal{F}}

\newcommand{\cJ}{\mathcal{J}}
\newcommand{\cK}{\mathcal{K}}

\newcommand{\cP}{\mathcal{P}}

\newcommand{\cR}{\mathcal{R}}

\newcommand{\fM}{\mathfrak{M}}

\newcommand{\wh}{\widehat}
\newcommand{\wt}{\widetilde}

\newcommand{\varep}{\varepsilon}

\newcommand{\im}{\operatorname{Im}}
\newcommand{\ind}{\mathds{1}}

\begin{document}

\title{discrete analogues of second-order Riesz transforms} 

\author{Rodrigo Ba\~nuelos}
\address{
Department of Mathematics \\ 
Purdue University \\ 
150 N.\@ University Street, West Lafayette, IN 47907
}
\email{banuelos@purdue.edu}

\author{Daesung Kim}
\address{
}
\email{daesungkim879@gmail.com}

\thanks{R.~Ba\~nuelos was supported in part by NSF Grant 1403417-DMS}

\begin{abstract}


Discrete analogues of classical operators in harmonic analysis have been widely studied, revealing deep connections with areas such as ergodic theory and analytic number theory. This line of research is commonly known as \emph{Discrete Analogues in Harmonic Analysis (DAHA)}.

In this paper, we study the $\ell^p$ norms of discrete analogues of second-order Riesz transforms. Using probabilistic methods, we construct a new class of second-order discrete Riesz transforms $\mathcal{R}^{(jk)}$ on the lattice $\mathbb{Z}^d$, $d \ge 2$. We show that for $1<p<\infty$, their $\ell^p(\mathbb{Z}^d)$ norms coincide with those of the classical second-order Riesz transforms $R^{(jk)}$ on $L^p(\mathbb{R}^d)$ when $j \neq k$, and are comparable up to dimensional constants when $j = k$.

The operators $\mathcal{R}^{(jk)}$ differ from the discrete analogue $R^{(jk)}_{\mathrm{dis}}$ by convolution with an $\ell^1(\mathbb{Z}^d)$ function. Applications are given to the DAHA of the Beurling--Ahlfors operator. We also show that $\mathcal{R}^{(jk)}$ arise as discrete analogues of certain Calder\'on--Zygmund operators $\mathbf{R}^{(jk)}$, which differ from $R^{(jk)}$ by convolution with an $L^1(\mathbb{R}^d)$ function. Finally, we conjecture that the $L^p$ norms of $\mathcal{R}^{(jk)}$, $R^{(jk)}_{\mathrm{dis}}$, and $\mathbf{R}^{(jk)}$ agree with those of the classical Riesz transforms, known to equal the corresponding martingale transform norms.

\end{abstract}

\maketitle
\tableofcontents 

\section{Introduction and statement of results}

The classical first-order Riesz transforms $R^{(k)}$, $k=1,2,\cdots,d$, are the basic Calder\'on-Zygmund singular integrals on $\R^d$, $d\geq 2$, extending in a natural way the Hilbert transform on $\R$. As singular integrals, they are defined by 
\begin{align*}
    R^{(k)}f(x) =  \int_{\R^d} K^{(k)}(z)f(x-z)\, dz,\qquad
    K^{(k)}(z) = \wt{c}_d\frac{z_k}{|z|^{d+1}},
\end{align*}
where $\wt{c}_d=\pi^{-\tfrac{d+1}{2}}\Gamma(\tfrac{d+1}{2})$ and the integration is in the principal value sense. They are  Fourier multipliers given by $\widehat{{R^{(k)}f}}(\xi)=\frac{i\xi_k}{|\xi|}\widehat{f}(\xi).$ They are also defined  by  
\begin{align}\label{firstRiesz} 
    R^{(k)}f(x)
    =\int_0^{\infty} \frac{\partial P_{y}  f(x)}{\partial x_k}\, dy
    = -\frac{\partial }{\partial x_k}(-\Delta)^{-1/2}f,
\end{align}
where $P_y f$ denotes the convolution of $f$ with the Poisson kernel 
\begin{align*}
    p(x,y) 
    = \frac{\wt{c}_d\,\,y}{\left(|x|^2 + y^2\right)^{\frac{d+1}{2}}},\quad  x\in \R^d, y>0,
\end{align*}
and $\Delta$ is the Laplacian on $\R^d$. Similarly, the second-order Riesz transforms $R^{(jk)}$ are defined as convolution operators with kernels 
\begin{align}\label{Ref2}
    K^{(jk)}(z)=c_d\frac{z_j z_k}{|z|^{d+2}},\quad j,k=1,2,\ldots,d, \quad
    c_d = \pi^{-\tfrac{d}{2}}\Gamma(\tfrac{d+2}{2}).
\end{align} 
In terms of the Fourier multiplier, they are given by 
\[
    \widehat{{R^{(jk)}f}}(\xi)=-\frac{\xi_j\xi_k}{|\xi|^2}\widehat{f}(\xi).
\]
If $T_t f$ denotes the convolution with the heat kernel
\[  
    P_t^{(d)}(x)=(2\pi t)^{-\tfrac{d}{2}}e^{-\tfrac{|x|^2}{2t}}, \quad x\in \R^d, t>0,
\]
then 
\begin{align}\label{secondRiesz} 
    R^{(jk)}f(x)
    =\int_0^{\infty} \frac{\partial^2  T_{t}  f(x)}{\partial x_j\partial x_k}\, dt
    =\frac{\partial^2 }{\partial x_j\partial x_k}(-\Delta)^{-1}f(x).
\end{align}

The formulations~\eqref{firstRiesz} and~\eqref{secondRiesz} allow the definition of Riesz transforms in a wide range of geometric settings, including manifolds, Lie groups, and Wiener space. These operators and their variants have been widely studied, with probabilistic methods playing a key role in obtaining sharp or near-sharp $L^p$-bounds; see~\cites{BanBau13, Dragi, BanOsc15, DomOscPet, DomPetSkr, Li08, Ose2} for an incomplete sample of this literature.

For the rest of the paper, $1<p<\infty$. We define  $p^*$ to be the maximum of $p$ and $q$,  where $q$ is the conjugate exponent of $p$. That is, 
\[
    p^\ast = \max\left\{p,\frac{p}{p-1}\right\}    
\]
so that 
\begin{equation}\label{BirkholdrConstant} 
    p^*-1= \begin{cases}
    \frac{1}{p-1}, &1<p\leq 2,\\
    p-1, & 2\leq p<\infty. 
    \end{cases}
\end{equation}
We will often refer to $(p^*-1)$ as ``Burkholder's constant.'' For the remaining of the paper the   $L^p$-norm of a function $f$ on  $\R^d$  with respect to the Lebesgue measure will be denote by $\|f\|_{L^p}$ and  the $p$-norm of a linear operator $T$ on $L^p(\R^d)$ will be denoted by $\|T\|_{L^p \to L^p}$.  That is, 
 \[
\|T\|_{p \to p}
   = \sup_{\|f\|_{L^{p}} = 1} \|Tf\|_{L^{p}}.
   \]
It was proved in~\cite{NazVol} and~\cite{BanMen} that for $1<p<\infty$,  $j,k=1,2,\ldots,d$, with $j\neq k$, 
\begin{align} 
    &\|2R^{(jk)}\|_{L^p\to L^p} \le (p^*-1),\label{RieszSharp1_1}\\
    &\|R^{(jj)}-R^{(kk)}\|_{L^p\to L^p}\leq (p^*-1). \label{RieszSharp1_2}
\end{align}
In~\cite{GesMonSak}, the sharpness of these bounds was proved. That is, 
\begin{align} 
    &\|2R^{(jk)}\|_{L^p\to L^p} =(p^*-1)\label{FullSharp}, \\
    &\|R^{(jj)}-R^{(kk)}\|_{L^p\to L^p}= (p^*-1)\label{FullSharp2},
\end{align} 
for $j\neq k$. 

When $j=k$, it was  proved in~\cite{BanOse} that 
\begin{equation}\label{RieszSharp2}
    \|R^{(jj)}\|_{L^p\to L^p}=\gamma(p),
\end{equation} 
where here and for the rest of the paper $\gamma(p)$ is the best constant given by Choi~\cite{Cho} for non-symmetric martingale transforms when the transforming predictable sequence takes values in $[0, 1]$ instead of in $[-1, 1]$, which is Burkholder's celebrated result~\cite{Bur84}. Although Choi's constant is not as explicit as Burkholder's constant $(p^*-1)$, it can be estimated quite well, as Choi showed. In particular, for large $p$ it satisfies 
\begin{align}\label{ChoiAsym}
    \gamma(p)\approx \frac{p}{2}+ \frac{1}{2}\log\left(\frac{1+e^{-2}}{2}\right) +\frac{\alpha_2}{p},
\end{align} 
where
\[
    \alpha_2=\left[\log\left(\frac{1+e^{-2}}{2}\right)\right]^2+\frac{1}{2}\log\left(\frac{1+e^{-2}}{2}\right)-2\left(\frac{e^{-2}}{1+e^{-2}}\right)^{2}.
\] 
Furthermore, for all $p$ 
\begin{align*}
    \max\left\{1, \frac{(p^*-1)}{2}\right\}\leq \gamma(p)\leq \frac{p^*}{2}. 
\end{align*} 
In particular when $p=2$,  both Burkholder's constant and  Choi's constant are equal to $1$.

For more general sharp inequalities that apply to multiplier operators related to those in~\eqref{FullSharp}-\eqref{RieszSharp2}, we refer the reader to~\cite{BanOse}*{Corollaries  1.3 and 1.4}. 

\subsection{Discrete analogues of singular integrals} 

The most basic singular integral is the Hilbert transform on the real line, defined for $f:\R\to\R$, in the principal value sense, by 
\begin{equation}\label{Hilbert1} 
    Hf(x)=\frac{1}{\pi}\int_{\R} \frac{f(x-y)}{y} dy.  
\end{equation} 
Its discrete analogue defined on the integers for $f:\Z\to \R$ is given by 
\begin{equation}\label{Hilbert2} 
    H_{\dis}f(n)=\frac{1}{\pi}\sum_{m\in\Z\setminus\{0\}}\frac{f(n-m)}{m}. 
\end{equation} 
The latter was introduced by D.~Hilbert in 1909. Its boundedness on $\ell^2(\Z)$ appeared in H.~Weyl's thesis written under Hilbert. In 1929, M.~Riesz~\cite{Riesz} proved its boundedness on $\ell^p(\Z)$ for any $1<p<\infty$ as a consequence of his proof for the boundedness of a continuous operator on $L^p(\R)$. The $L^p$-boundedness of $H$  was a problem of considerable interest to the analysis community at the time. See for example~\cite{Car1982}. For other variants of the discrete Hilbert transform and further history and references, we refer the reader to~\cite{BanKwa1, Titc26, Titc27, Laeng, Tor} and to \cite{BanKwa} where the long-standing,  90-year-old conjecture that $\|H\|_{L^p\to L^p}=\|H_{\dis}\|_{\ell^p\to \ell^p}$, is proved. Here and for the rest of the paper $\|\cdot\|_{\ell^p}$ and $\|\cdot\|_{\ell^p\to \ell^p}$ denote, respectively,  the $p$-norm of a function on the $d$-dimensional lattice $\Z^d$ and the $p$-norm of an operator on $\ell^p(\Z^d)$.

In their 1954 seminal paper, Calder\'on and Zygmund~\cite{CZ} extended Hilbert's definition to discrete singular integrals on the $d$-dimensional lattice $\Z^d$, $d> 1$. They observed~\cite[pg.~138]{CZ} that Riesz's proof that the $L^p$-boundedness of the Hilbert transform on the real line implies the $\ell^p$-boundedness of the discrete version on the integers applies to general singular integrals on $\R^d$, $d>1$. More precisely, let $Tf(x)=K*f(x)$ be a Calder\'on-Zygmund singular integral on $\R^d$ with kernel $K$ satisfying 

\begin{align}\label{ConCalZyg1}
\begin{cases}
(i)\,\,\,\,\,   \wh{K}\in L^{\infty}(\R^d\setminus\{0\}),&\\
(ii)\,\,\,     |K(x)|\leq \frac{\eta}{|x|^d}, & x\in \R^d\setminus\{0\}, \\
(iii)\,        |\nabla K(x)|\leq \frac{\eta}{|x|^{d+1}}, &x\in \R^d\setminus\{0\}. 
\end{cases}
\end{align}
For $f:\Z^d\to \R$, define its discrete analogue by
\begin{align}\label{DiscCalZyg}
    T_{\dis}f(n)=\sum_{m\in\Z^d}K_{\dis}(m)f(n-m),\quad K_{\dis}(m)=K(m)\ind_{\Z^d\setminus\{0\}}(m). 
\end{align}
Then 
\[T_{\dis}:\ell^p(\Z^d)\to\ell^p(\Z^d).
\] 
In fact, Riesz's argument yields  
\begin{equation}\label{RieszCZ}
    \|T_{\dis}\|_{\ell^p\to\ell^p}\leq  \|T\|_{L^p\to L^p}+C_{d, \eta}, 
\end{equation} 
where $C_{d, \eta}$ is a constant depending on the dimension $d$ and on the constant $\eta$ in~\eqref{ConCalZyg1}. 

Although no details of the proof of~\eqref{RieszCZ} are given--the only remark in~\cite[pg.~138]{CZ} (where their $n$ corresponds to our $d$) is that “for $n=1$ this remark is due to M. Riesz, and the proof in the general case follows a similar pattern”--the argument is indeed the same as Riesz’s, and full details have been known for many years. For further results and references in the area commonly known as Discrete Analogues in Harmonic Analysis (DAHA), which concerns the study of discrete counterparts on $\Z^d$ of operators in harmonic analysis on $\R^d$, we refer the reader to~\cites{MSW,PieSte1,Pierce,Kra}.

In this paper, motivated by the recent results in~\cite{BanKwa,BanKwa1,BanKimKwa} concerning sharp estimates for DAHA of the Hilbert and first-order Riesz transforms, we investigate similar problems for  the DAHA of second-order Riesz transforms defined as in \eqref{DiscCalZyg} by  
\begin{align}\label{disRiesz1} 
    R^{(jk)}_{\dis}(f)(n)=\sum_{m\in\Z^d}K^{(jk)}_{\dis}(m)f(n-m),
\end{align}
where
\begin{align}\label{disRiesz2}
    K^{(jk)}_{\dis}(m) = c_d \frac{m_j m_k}{|m|^{d+2}}\ind_{\Z^d\setminus\{0\}}(m), \qquad
    c_d =\pi^{-\tfrac{d}{2}} \Gamma(\tfrac{d+2}{2}).
\end{align}
Notice that since $|K^{(jk)}_{\dis}(m)|\leq \frac{c_d}{|m|^d}$, it follows from H\"older's inequality that the sum in \eqref{disRiesz1} is absolutely convergent for $f\in \ell^p$, $1<p<\infty$. Similarly, absolute convergence holds for the general discrete singular integrals with Calder\'on-Zygmund kernels as in \eqref{ConCalZyg1} and even more general singular integrals of non-convolution type \cite[Chapter 8]{Graf}.

\subsection{Main results} 
This paper deals with four distinct types of second-order Riesz transforms, with the overarching goal of showing that their $p$-norms coincide.

\begin{enumerate}
\item 
{\bf Classical second-order Riesz transforms.}   These are the classical second-order Riesz transforms $R^{(jk)}$ on $L^p(\R^d)$ with kernels given as in  \eqref{Ref2}.  
\item 
{\bf Discrete harmonic analysis analogues (DAHA) of second-order Riesz transforms.}  These are the discrete operators  $R_{\dis}^{(jk)}$ on $\ell^p(\Z^d)$ with kernels given as in  \eqref{DiscCalZyg}. 
\item 
{\bf Probabilistic discrete second-order Riesz transforms.} This is a new class of discrete  operators, denoted by $\cR^{(jk)}$, on $\ell^p(\Z^d)$ with kernels given by \eqref{ProKer}. They arise as conditional expectations of martingale transforms constructed from Doob $h$-processes, where $h$ is the periodic heat kernel $H_t(x)$ defined in \eqref{def_h}. 
\item 
{\bf Probabilistic continuous second-order Riesz transforms}. This is a new class of operators on $L^p(\R^d)$, denoted by  $\mathbf{R}^{(jk)}$, with kernels given by \eqref{ContinousKernels}. The kernels  are defined, formally, by reversing the construction in the definition of DAHA operators. For $|x|>1$, the kernels are defined to be the the kernels of $\cR^{(jk)}$ with $m$ replaced by $x$. For $|x|\leq 1$, the kernels are just the kernels for classical Riesz transforms on $\R^d$. It is then verified that the resulting kernels satisfy the Calder\'on-Zygmund conditions \eqref{ConCalZyg1}.  Hence, the probabilistic discrete second-order Riesz transforms are the DAHA of the probabilistic continuous second-order Riesz transforms in the sense of \eqref{DiscCalZyg}.
\end{enumerate} 

The problems addressed in this paper are the identification their $p$-norms and the Conjectures \ref{BigConj} and \ref{ProbContSharp} that they are all equal. We prove  that the $\ell^p(\Z^d)$-norms of $\cR^{(jk)}$ coincide with the $L^p(\R^d)$-norms of $R^{(jk)}$ for all $j \neq k$, and are comparable to it up to dimensional constants when $j = k$. In addition, it is shown that for all $j, k$, $\cR^{(jk)}$ differ from $R^{(jk)}_{\dis}$ by convolution with an $\ell^1(\Z^d)$-function. Similarly, the continuous operators $\mathbf{R}^{(jk)}$ and $R^{(jk)}$ differ by convolution with an $L^1(\R^d)$-function.

The construction of the probabilistic Riesz transforms and bounds for their $\ell^p(\Z^d)$-norms is motivated by \cite{BanMen}, where the space-time Brownian motion in the upper half space was used to represent the classical second-order Riesz transforms as conditional expectations of martingale transforms.  It also follows from, and is motivated by,  the discretization procedure used  in~\cite{BanKwa, BanKimKwa} for the discrete Hilbert transform and first-order Riesz transforms. Here, the periodic Poisson kernel used in the last two papers is replaced by the periodic heat kernel. The martingale inequalities of Burkholder~\cite{Bur84} and Choi~\cite{Cho} are used to derive upper bounds. 

The precise statements on the norms of the probabilistic Riesz transforms are: 

\begin{theorem}\label{thm:main} 
For all  $j,k=1,2,\ldots,d$. with $j\neq k$, we have 
\begin{align}\label{off-dia1}
    \|2\cR^{(jk)}\|_{\ell^p\to\ell^p}= (p^*-1)
    =\|2R^{(jk)}\|_{L^p\to L^p} \le \|2R_{\dis}^{(jk)}\|_{\ell^p\to\ell^p}
\end{align}
and 
\begin{align}\label{off-dia2}
  \|\cR^{(jj)}-\cR^{(kk)}\|_{\ell^p\to\ell^p}
   &=(p^*-1)
 =\|R^{(jj)}-R^{(kk)}\|_{L^p\to L^p}\\
    & \le \|R_{\dis}^{(jj)}-R_{\dis}^{(kk)}\|_{\ell^p\to\ell^p}\nonumber.
\end{align}
\end{theorem}

\begin{theorem}\label{thm:diag}
For all $j=1,2,\ldots,d$,  
\begin{align}\label{equal:diag}
    \|\cR^{(jj)}\|_{\ell^p\to\ell^p}\le \gamma(p)=\|R^{(jj)}\|_{L^p\to L^p}.   
\end{align}
 Furthermore, there exist dimensional constants $C_1$ and $C_2$ such that 
\begin{align}
    C_1\|R^{(jj)}_{\dis}\|_{\ell^p\to\ell^p}
    \le \|\cR^{(jj)}\|_{\ell^p\to\ell^p}
    \le C_2\|R^{(jj)}_{\dis}\|_{\ell^p\to\ell^p}.
\end{align} 
\end{theorem} 

\begin{remark} 
We observe that the second equalities in~\eqref{off-dia1},~\eqref{off-dia2} and~\eqref{equal:diag} are in fact the equalities in~\eqref{FullSharp},~\eqref{FullSharp2} and~\eqref{RieszSharp2}.  It will be shown in Section~\ref{ProbDisRiesz} (Corollary~\ref{HalfSharp}), that $``\leq"$ in the first equalities of~\eqref{off-dia1},~\eqref{off-dia2}, as well as the inequality in~\eqref{equal:diag},  follow from the martingale inequalities of Burkholder and Choi. The proof of the opposite inequalities in~\eqref{off-dia1} and~\eqref{off-dia2} is done in Propositions ~\ref{discont1} and \ref{lem:S-pnorm} using a dilation and approximation argument. Such an argument does not apply to the kernels for $\cR^{(jj)}$;  see the comments in the introduction of Section \ref{sec:proof} and Remark~\ref{rmk:whyjnek}.  
\end{remark}

\begin{conjecture}\label{BigConj}
For all $j, k=1,2,\ldots,d$, $j\neq k$, 
\begin{align}
    \|\cR^{(jk)}\|_{\ell^p\to\ell^p}
    =\|R_{\dis}^{(jk)}\|_{\ell^p\to\ell^p}
    =\|R^{(jk)}\|_{L^p\to L^p}, 
\end{align} 
and
\begin{align}
    \|\cR^{(jj)}-\cR^{(kk)}\|_{\ell^p\to\ell^p}
    =\|R_{\dis}^{(jj)}-R_{\dis}^{(kk)}\|_{\ell^p\to\ell^p}
    =\|R^{(jj)}-R^{(kk)}\|_{L^p\to L^p}. 
\end{align}
For all $j=1, 2, \dots d$, 
\begin{align}
    \|\cR^{(jj)}\|_{\ell^p\to\ell^p}
    =\|R_{\dis}^{(jj)}\|_{\ell^p\to\ell^p}
    =\|R^{(jj)}\|_{L^p\to L^p}.
\end{align} 
\end{conjecture}

Notice that~\eqref{FullSharp},~\eqref{FullSharp2}, and the Calder\'on-Zygmund inequality ~\eqref{RieszCZ} give 
\begin{align}\label{Upper1-CZ}
    \|2R_{\dis}^{(jk)}\|_{\ell^p\to\ell^p}&\leq \|2R^{(jk)}\|_{L^p\to L^p}+C_d\\
    &=(p^*-1)+C_d,\,\,\, j\neq k,\nonumber
\end{align}
and
\begin{align}\label{Upper2-CZ}
    \|R_{\dis}^{(jj)}-R_{\dis}^{(kk)}\|_{\ell^p\to\ell^p}&\leq \|R^{(jj)}-R^{(kk)}\|_{L^p\to L^p}+C_d\\
    &=(p^*-1)+C_d, \,\,\, j\neq k.\nonumber
\end{align}

In \eqref{convolution1}, \eqref{convolution2} and Corollary~\ref{KerJ}, we show that the operators $\cR^{(jk)}$ (or rather $-\cR^{(jk)}$) differ from the operators $R_{\dis}^{(jk)}$  by the convolution with an $\ell^1(\Z^d)$-function whose $\ell^1$-norm is bounded above by a dimensional constant $C_d$.  More precisely, for $f:\Z^d\to \R$ and for all $j, k$, 
\[
    (R_{\dis}^{(jk)}+\cR^{(jk)}) f(n)=\cJ^{(jk)}* f(n), 
\]
where $\|\cJ^{(jk)}\|_{\ell^1}\leq C_d$.  This yields the bounds  \eqref{Upper1-CZ} and \eqref{Upper2-CZ}  without appealing to Calder\'on-Zygmund theory, see Corollary~\ref{J-kern}.

Although the function $\cJ^{(jk)}$ has an appealing closed form (see Corollary~\ref{InverseGammaU}, the proof of Corollary \ref{U-new} and \eqref{Referee1} in the proof of Corollary~\ref{KerJ}), we have not been able to evaluate  its $\ell^1$-norm. In particular, it would be interesting to prove  that $\|\cJ^{(jk)}\|_{\ell^1}\leq C$, with $C$ independent of $d$.  This  will show that the $\ell^p$-norms of the operators $R^{(jk)}_{\dis}$ are bounded above with a bound of the forms $C(p^*-1)$,  with $C$ independent of $d$. 
Such a result would be weaker than Conjecture~\ref{BigConj} but of interest nevertheless. To the best of our knowledge, even verifying that $R^{(j)}_{\dis}$ or $R^{(jk)}_{\dis}$  are bounded on $\ell^2(\Z^d)$ with constants independent of dimension is not currently known. Perhaps calculations similar to those in \cite{Gra1} to compute the $\ell^2$ norm of $H_{\dis}$  might yield the desired dimension-free bounds in this case. 

Finally, it is important to note here that the discrete second-order Riesz transforms defined by periodic Fourier multipliers, as in \cite{deLeeuw}, \cite[Section 2]{MSW}, \cite[Chapter VII]{SteWei}, and many other available references on transference techniques, have been extensively studied for many years. A similar statement applies to Riesz transforms obtained from discrete derivatives and semi-discrete Laplacian, where techniques from Littlewood-Paley theory and the martingale approach  of Gundy-Varopoulos~\cite{GV79} have been successfully used to obtain sharp or near-sharp bounds on their $p$-norms. For some of these applications, see, for example,~\cite{Pet3, DomOscPet, Tor, Lus-Piq4} and the references therein.  

\subsection{Structure of paper} 
The paper is organized as follows. In Section~\ref{sec:probdisc}, we establish several properties of the periodic heat kernel and its connection with the Jacobi theta function. We also recall the martingale inequalities used throughout the paper and construct the general conditional expectation operators, identifying their corresponding kernels. In Section~\ref{ProbDisRiesz}, we define the probabilistic second-order Riesz transforms and show that their kernels are multiples of those of $R^{(jk)}_{\dis}$ by a function that--although somewhat complicated--admits an explicit expression and has  useful boundedness properties. The proofs of Theorems~\ref{thm:main} and~\ref{thm:diag} are given in Section~\ref{sec:proof}. The construction of the probabilistic discrete Beurling-Ahlfors operator is presented in Section~\ref{sec:BA}. Finally, in Section~\ref{CZ-Prob}, we answer in the affirmative the natural question of whether the probabilistic second-order discrete Riesz transforms $\cR^{(jk)}$ arise as DAHA of Calder\'on-Zygmund singular integrals  similarly to~\eqref{disRiesz1} and~\eqref{disRiesz2}.

Throughout the paper, we use $C_0, C_1, C_2, \ldots$, or simply $C_d$, to denote constants that depend on the dimension $d$ and whose values may change from line to line. In contrast, $c_d$ will always denote the normalization constant appearing in the kernel of the classical second-order Riesz transforms given in~\eqref{Ref2}.  

\section{Space-time martingales and discrete singular integrals}\label{sec:probdisc}

This section constructs a large collection of discrete singular integrals on the lattice $\Z^d$ which are bounded on $\ell^p(\Z^d)$, $1<p<\infty$, with explicit upper bounds on their norms. They are obtained as conditional expectations of martingale transforms of a space-time process in the style of the construction in \cite{BanMen} and ~\cite{BanKwa, BanKimKwa}.  We will see in Section~\ref{ProbDisRiesz} that this construction leads to the probabilistic discrete second-order Riesz transforms whose $\ell^p(\Z^d)$-norms coincide with those of the classical Riesz transforms on $L^p(\R^d)$--Theorem~\ref{thm:main}. The periodic Poisson kernel used in~~\cite{BanKwa, BanKimKwa} is replaced here by the periodic heat kernel to make the connection to~\cite{BanMen}. 

\subsection{The periodic heat kernel}

The heat kernel on $\R^d $, $d\geq 1$, is given by
\[
    P_t^{(d)}(x)=(2\pi t)^{-\tfrac{d}{2}}e^{-\tfrac{|x|^2}{2t}},\qquad x\in \R^d, t>0. 
\]
Note that with $x=(x_1, x_2, \dots, x_d)$, $P^{(d)}_t(x) = \prod_{k=1}^d P^{(1)}_t(x_k)$ and we simply write $P_t( x):=P^{(d)}_t(x)$. Define
\begin{equation}\label{def_h}
   H_t(x)= H(x,t)=\sum_{n\in \Z^d}P_t(x-n),\quad x\in\R^d,  t> 0. 
\end{equation}
We call this the {\it periodic heat kernel in $\R^d$}, which is the same as the heat kernel on the $d$-dimensional torus  which for the rest of the paper will be denoted by $Q$ and identified as $Q=[-\tfrac{1}{2}, \tfrac{1}{2})^d$.  In this paper, to make the connection to second-order Riesz transforms, we will be interested in the case of  $d\geq 2$.  For notational convenience at various points in what follows we will use either $H_t(x)$ or $H(x,t)$ for the function $H$. 

It is well-known that the heat kernel on the torus can be expressed in terms of the first Jacobi Theta function, see for example~\cite[Chapter 10]{SteSha}. For $z\in\bC$ and $\tau\in\bH:=\{\tau\in\bC: \im(\tau)>0\}$, the {\it Jacobi Theta function} is given by
\[
    \vartheta(z\mid\tau) = \sum_{n=-\infty}^\infty \exp(\pi i n^2 \tau + 2\pi i n z).
\]
The function $\vartheta$ is entire in $z\in\bC$, holomorphic in $\tau\in\bH$, and  the following identities hold: 
\begin{align*}
    \vartheta(z\mid\tau) &= \vartheta(-z\mid\tau), \qquad
    \vartheta(z+1\mid\tau) = \vartheta(z\mid\tau), \\
    \vartheta(z\mid -1/\tau) &= \sqrt{\frac{\tau}{i}}e^{\pi i \tau z^2}\vartheta(z\tau\mid\tau).
\end{align*}
Here we use the variant of $\vartheta$ given by 
\begin{align*}
    \theta(x,t):=\vartheta(x\mid 2\pi i t)=\sum_{n\in\Z}\exp(-2t\pi^2 n^2+2\pi i n x), \quad x\in\R,\,\,  t>0.
\end{align*}
By the  Poisson summation formula,  
\begin{align}\label{eq:h-prod}
   H_t(x)= \prod_{k=1}^d\theta(x_k, t),
\end{align}
for $x\in\R^d$, $t>0$. The representation of $H_t$ will be useful in various calculations below. 

Our probabilistic Riesz transforms $\cR^{(jk)}$ will be constructed from conditional expectations of martingale transforms arising from a space-time Doob $h$-process where the function $h$ is the periodic heat kernel $H_t(x)$. For this purpose, we need various estimates not only of $H_t(x)$ but also of its gradient. 

The literature on heat kernel estimates on manifolds--including compact manifolds with nonnegative curvature of which the torus is a very special case--is extensive. It is therefore quite likely that the estimates we require have already appeared in, or can be deduced from, existing results in the literature. To make the paper self-contained and to avoid technicalities that lie outside the  main focus of this paper and unless the results are explicitly sated for the torus,  we provide direct proofs of the necessary estimates. Our arguments rely on elementary calculations using the explicit formula for $H_t(x)$.

We begin by recalling the upper and lower Gaussian estimates for $H_t(x)$ which can be found in \cite[Sections 2 \& 3]{Mah}. For all $x\in \R^d$ and all $t>0$, 
\begin{align}\label{heatestimate}
    (2\pi t)^{-\tfrac{d}{2}}e^{-\frac{\|x\|_{Q}^2}{2t}}\leq H_t(x)\leq C_d (t\wedge 1)^{-\frac{d}{2}}e^{-\frac{\|x\|_{Q}^2}{2t}},
\end{align} 
where $\|x\|_Q=\min\{|x-n|:n\in\Z^d\}$.     In fact, the lower bound follows trivially by observing that 
\[ 
    (2\pi t)^{-\tfrac{d}{2}}e^{-\frac{\|x\|_{Q}^2}{2t}}=\sup_{n\in \Z^d}P_t(x-n)\leq \sum_{n\in \Z^d} P_t(x-n)=H_t(x). 
\]
The upper bound in \eqref{heatestimate} is based on the explicit formula for $H_t$ and again its proof is of an elementary nature. We also note that for all $x\in Q$, $0\leq \|x\|_Q\leq \sqrt{d}/2$. 

The following lemma will be used to identify the boundary values of the Doob $h$-parabolic function $u_f(x, t)$ defined by \eqref{eq:def_u} and to establish  the bounds for the function $U(m)$  appearing  in Theorem \ref{prop:Kdecomp} (see Corollary \ref{lem:Ulimit}).  
\begin{lemma}\label{lem:hlimit}
For all $(x,t)\in\R^d\times(1,\infty)$, $|H_t(x)-1|\le \frac{C_d}{\sqrt{t}}$.  In particular,  $\lim_{t\to\infty}H_t(x)=1$  uniformly in $x\in\R^d$. Furthermore, for all $x\in\R^d$, 
\begin{align}\label{ConstantHeat}
    \left| \frac{1}{H_t(x)}-1\right| 
    \le 
    \begin{cases}
        \frac{C_1}{\sqrt{t}} e^{\frac{C_2}{t}}, & 0<t<1,\\
        \frac{C_3}{\sqrt{t}}, & t\ge 1,
    \end{cases}
\end{align}
for some dimensional constants $C_1,C_2,C_3>0$.  
\end{lemma}

\begin{proof}
Since  $H_t(x)$ is periodic in $x$, it suffices to consider $x\in Q$. Since
\begin{align*}
    1=\int_{\R^d} P_t(x-z)\, dz = \sum_{n\in \Z^d}\int_{n+Q}P_t(x-z)\, dz
    =\sum_{n\in\Z^d} \int_Q P_t(x-z-n)\, dz,
\end{align*}
we have
\begin{align*}
    H_t(x)-1
    = \sum_{n\in\Z^d} \int_Q (P_t(x-n)-P_t(x-z-n))\,dz.
\end{align*}
For $x\in\R^d$ and $y\in Q$, using $|x-y|^2\ge \frac12 |x|^2-|y|^2\ge \frac12 |x|^2-\frac{d}{4}$ and
\[
    |x-y| e^{-\frac{|x-y|^2}{2}}\le \sqrt{2}e^{-\frac12} e^{-\frac{|x-y|^2}{4}},
\]
we obtain
\[
    |\nabla_x P_t(x-y) | 
    = \tfrac{|x-y|}{t}P_t(x-y) 
    \le \tfrac{C_d}{\sqrt{t}}P_{2t}(x-y)
    \le  \tfrac{C_d}{\sqrt{t}}P_{4t}(x)e^{\frac{d}{16t}}.
\]
Thus, by the mean value theorem, 
\begin{align}\label{L1-vs-L2}
    |H_t(x)-1|
    &\le \sum_{n\in\Z^d} \int_Q |z||\nabla_x P_t(x-n-\theta z)|\,dz\nonumber\\
    &\le \frac{C_d}{\sqrt{t}}e^{\frac{d}{16t}}\sum_{n\in\Z^d} \int_Q P_{4t}(x-n)\,dz\nonumber\\
    &=\frac{C_d}{\sqrt{t}}H_{4t}(x)e^{\frac{d}{16t}}.
\end{align}

Suppose $(x,t)\in\R^d\times (1,\infty)$. By~\eqref{heatestimate}, we have $|H_{4t}(x)|\le C_d$. Since $e^{\frac{d}{16t}}\le C_d$ for $t\ge 1$, we conclude that 
\[
    |H_t(x)-1|\le \frac{C_d}{\sqrt{t}}, 
\]
proving the first assertion in the Lemma. 

For \eqref{ConstantHeat}, consider $(x,t)\in\R^d\times (0,\infty)$. Using \eqref{heatestimate} we have that 
\[
    \frac{H_{4t}(x)}{H_t(x)}
    \leq C_d \frac{(4t\wedge 1)^{-\frac{d}{2}}e^{-\frac{\|x\|_{Q}^2}{8t}}}{t^{-\frac{d}{2}}e^{-\frac{\|x\|_{Q}^2}{2t}}}
    \le C_1 e^{\frac{C_2}{t}}
\]
and hence
\[
    \left|\frac{1}{H_t(x)}-1\right|
    \le \frac{C_1}{\sqrt{t}}\frac{H_{4t}(x)}{H_t(x)}e^{\frac{d}{16t}}
    \le \frac{C_1}{\sqrt{t}}e^{\frac{C_2}{t}}, 
\]
where $C_1,C_2>0$ are dimensional constants. We finish the proof by taking $C_3=C_1e^{C_2}$ for the case $t\ge 1$.
\end{proof}

The following lemma will be used in Section \ref{ProjectionOperators} to compute the kernels for the  discrete operators arising from conditional expectations of martingale transforms.  For the rest of the paper, the gradient $\nabla$ and the Laplacian $\Delta$  are with respect to the $x$ variable, and $\partial_{z} \theta(x_k,t)$ will denote the derivative of $\theta$ with respect to the first variable. 

\begin{lemma}\label{lem:der-h-bound} For all $x\in \R^d$ and $t>0$ there exist dimensional constants $C_1,  C_2$ such that
\[|\nabla\log H_t(x)|\le 
    \begin{cases}
        \frac{C_1}{t}, &0<t<1, \\
        \frac{C_2}{\sqrt{t}}, &t\ge 1.
    \end{cases}
\]    
\end{lemma}

\begin{proof}
Since $H_t(x)$ is periodic in $x\in\R^d$, we assume that $x\in Q=[-\tfrac{1}{2}, \tfrac{1}{2})^d$. By~\eqref{eq:h-prod}, we have
\[\left|\nabla\log H_t(x)\right|^2
    =\sum_{k=1}^d \left|\frac{\partial_{z} \theta(x_k,t)}{\theta(x_k,t)}\right|^2,
\]
where 
\begin{align}\label{eq:theta}
    \theta(x_k,t)
    =\sum_{n\in\Z}\exp(-2t\pi^2 n^2+2\pi i n x_k)
    =\sum_{n\in\Z}(2\pi t)^{-\frac12}\exp\left(-\frac{1}{2t}|x_k-n|^2\right).
\end{align}

{\bf\noindent Case 1:} Suppose $t\ge 1$. Then, by~\eqref{eq:theta},
\begin{align}\label{eq:thetalowerbound}
    \theta(x_k,t) 
    =|\theta(x_k,t)| 
    \ge 1-2\sum_{n=1}^\infty e^{-2\pi^2 n^2 t}
    \ge 1-2\sum_{n=1}^\infty e^{-2\pi^2 n}
    = \frac{e^{2\pi^2}-3}{e^{2\pi^2}-1}>0
\end{align}
and
\begin{align*}
    |\partial_{z} \theta(x_k,t)|
    \le 4\pi\sum_{n=1}^\infty n e^{-2\pi^2 n^2 t}
    = 4\frac{\pi}{\sqrt{t}}\sum_{n=1}^\infty (n\sqrt{t}) e^{-2\pi^2 (n\sqrt{t})^2 }.
\end{align*}
Since the function $u\mapsto ue^{-2\pi^2 u^2}$ is decreasing on $[1,\infty)$, $n\sqrt{t}\ge 1$ for all $n\ge 1$ and $t\ge 1$, we have
\[
    |\partial_{z} \theta(x_k,t)|
    \le \frac{4\pi}{\sqrt{t}} \int_0^\infty ue^{-2\pi^2 u^2}\, du
    =\frac{1}{\pi \sqrt{t}}.
\]
Thus, we obtain
\begin{align*}
    \left|\nabla \log H_t(x)\right|
    \le \frac{C_{d}}{\sqrt{t}}.
\end{align*}

{\bf\noindent Case 2:} Let $0<t<1$. 
Using $\theta(x_k,t)\ge (2\pi t)^{-\frac12}\exp\left(-\frac{1}{2t}|x_k|^2\right)$ and $|x_k-n|\le 2|n|$, we get
\begin{align*}
    \left|\frac{\partial_{z} \theta(x_k,t)}{\theta(x_k,t)}\right|
    &\le \frac{|x_k|}{t} + \frac{1}{t}\sum_{n\in\Z,n\neq 0} |x_k-n|\exp\left(-\frac{1}{2t}(n^2-2n x_k)\right)\\
    &\le \frac{|x_k|}{t} + \frac{2}{t}\sum_{n\in\Z,n\neq 0} |n|\exp\left(-\frac{1}{2t}(|n|^2-|n|)\right).
\end{align*}
Since
\begin{align*}
    \sum_{n\in\Z,n\neq 0} |n|\exp\left(-\frac{1}{2t}(|n|^2-|n|)\right)
    &=2+2\sum_{n=2}^\infty n\exp\left(-\frac{1}{2t}(n^2-n)\right)\\
    &\le 2+2\sum_{n=2}^\infty n\exp\left(-\frac{n^2}{4t}\right)\\
    &= 2+2t\sum_{n=2}^\infty \frac{n}{\sqrt{t}}\exp\left(-\frac14\left(\frac{n}{\sqrt{t}}\right)^2\right)\frac{1}{\sqrt{t}}\\
    & \le 2+2t\int_2^\infty u e^{-\frac{u^2}{4}}\, du\le 2+4t,
\end{align*}
we conclude that 
\[
    \left|\frac{\partial_{x_k} \theta(x_k,t)}{\theta(x_k,t)}\right|
    \le \frac{5}{t}+8\le \frac{13}{t},
\]
which finishes the proof. 
\end{proof}

\subsection{The heat extension}\label{heat-extension}  

Let $f:\Z^d \to\R$ be a compactly supported function. Define  
\begin{align}\label{eq:def_u}
    u_{f}(x,t)=\sum_{n\in \Z^d}f(n)\frac{P_{t}(x-n)}{H_t(x)}.
\end{align}
Since $H_t(x)u_f(x,t)=H(x,t)u_f(x,t)$ is a solution to the heat equation, we have
\begin{align*}
    \frac{1}{2}\Delta u_f+\frac{\nabla H\cdot \nabla u_f}{H} = \partial_t u_f.
\end{align*}

The following proposition provides information on the boundary values of the function $u_f(x, t)$.  In particular, the fact that the boundary values of $u_f(x, t)$ are in $L^p$, and agree with the function $f$ on the lattice $\Z^d$ plays a crucial role in what follows. In conjunction with Lemma \ref{boundary-values2} it shows that the martingale $M_t^f$ defined in \eqref{eq:def_M_t} has the property that $M_T^f=f(X_T)$, which is used in the Theorem \ref{thm:MartTrans}. Note also that the proposition is an analogue of the extension result in~\cite[Lemma 2.1]{MSW} with our bound here being less than or equal to 1.   
 
\begin{proposition}\label{boundary-values1}
For all $x\in\R^d$, $P_t(x)/H_t(x)$ converges as $t\to 0$. Let $\Psi(x)=\lim_{t\to0}P_t(x)/H_t(x)$ and $f\in\ell^p(\Z^d)$ be compactly supported. Define 
\[
    f_\ext(x) := \sum_{m\in\Z^d} f(n)\Psi(x-m).
\]
Then, $\|f_{\ext}\|_{L^p(\R^d)} \le\|f\|_{\ell^p(\Z^d)}$ and $f_\ext(n) = f(n)$,  for all $n\in\Z^d$.   
\end{proposition} 

\begin{proof}
First, we have
\begin{align*}
    H_t(0) 
    = P_t(0)+\sum_{m\neq 0}P_t(m)
    = P_t(0)\left(1+\sum_{m\neq 0}\exp\left(-\frac{|m|^2}{2t}\right)\right).
\end{align*}
Thus,
\[
    \frac{P_t(0)}{H_t(0)}=\frac{1}{1+\sum_{m\neq 0}\exp(-\frac{|m|^2}{2t})}\le 1
\]
and so $\lim_{t\to 0}\frac{P_t(0)}{H_t(0)}=1$. For $n\in\Z^d$ with $n\neq 0$, we have $\frac{P_t(n)}{H_t(n)}  \le  \frac{P_t(n)}{P_t(0)}  =e^{-\frac{|n|^2}{2t}} \to 0, $ as $t\to 0$. For $x\in\R^d\setminus\Z^d$, we have
\[
    \frac{P_t(x)}{H_t(x)} =\frac{1}{1+\sum_{m\neq 0}\exp\left(-\frac{1}{2t}(|m|^2-2x\cdot m)\right)}.
\]
By the monotone convergence theorem, we conclude that $\Psi(x)$ exists for all $x$.  Since
\begin{equation}\label{BM-Rep}
    \int_{\R^d}\frac{P_t(x)}{H_t(x)}\, dx
    =\sum_{m\in\Z^d}\int_{Q}\frac{P_t(x-m)}{H_t(x-m)}\, dx = 1,
\end{equation} 
we have 
\[
    \int_{\R^d}\Psi(x)\, dx\le 1\quad \text{and}\quad \sum_{m\in\Z^d}\Psi(x-m)\le 1,
\]
by Fatou's lemma. Thus, 
\begin{align*}
    \|f_{\ext}\|_{L^p}^p
    &=\int_{\R^d}\bigg|\sum_{m\in\Z^d} f(m)\Psi(x-m)\bigg|^p\, dx\\
    &\le \int_{\R^d}\bigg(\sum_{m\in\Z^d} |f(m)|^p\Psi(x-m)\bigg) \bigg(\sum_{m\in\Z^d} \Psi(x-m)\bigg)^{p-1}\, dx\\
    &\le \sum_{m\in\Z^d} |f(m)|^p = \|f\|_{\ell^p}^p,
    \end{align*}
as desired. 
\end{proof}

\subsection{Space time martingale transforms}\label{doob-parabolic} 

Let $T>0$ be fixed. Let $X_t$ be a solution of the stochastic differential equation 
\begin{align*}
    dX_{t}=dB_{t}+\nabla \log{H(X_{t},T-t)}dt,\quad X_0=0, 
\end{align*}
for $0\le t<T$, where $B_t$ is the standard $d$-dimensional Brownian motion. The solution to this equation is a  Doob $h$-process  conditioned by the parabolic function $H_t$.  Doob $h$-processes  have been extensively studied in the literature, both for the case of harmonic and parabolic functions $h$. We refer the reader to  Doob \cite[ 2.X]{Doob}, and Durrett \cite{Durrett} and Pinsky \cite{Pinsky},  for some of the classical applications to potential theory and other areas of analysis and partial differential equations. As already noted, the motivation and idea for the present Doob “$h$”-construction originates in the following two papers:  \cite{BanMen} where  “heat” martingales were employed to represent the classical second-order Riesz transforms $R^{(jk)}$ and  \cite{BanKimKwa} where conditioning by the periodic Poisson kernel was used to obtain a representation for the discrete probabilistic first-order Riesz transforms $\cR^{(k)}$. These constructions are all motivated by the Gundy-Varopoulos \cite{GV79} construction widely used in the study of Riesz transforms in many geometric settings.  

Consider the space-time process  $Z_t=(X_t,T-t)$ for $0\le t < T$.  

In what follows, we will use $\P_{(0, T)}$ and $\bE_{(0,T)}$ for the probability and expectation associated with the process $Z_t$ starting at the point $(0, T)$, $0\in \R^d$ and $T>0$. We recall that $\frac{P_t(x)H_{T-t}(x)}{H_T(0)}$ is the density of the random variable $X_t$ under the measure  $\P_{(0, T)}$. For this and other properties of the solution to the above type of SDE's, including its generator and the Fokker-Planck equation satisfied  by its density, see for example  \cite{Oks, Eva,RogWil, BurSal}. Define $Z_T:=\lim_{t\to T}Z_t$. The next lemma shows that this limit exists, that the random variable $X_T$ takes values in $\Z^d$ and it identifies its distribution under the probability measure $\P_{(0, T)}$.  

\begin{lemma}\label{boundary-values2}
For all $n\in\Z^d$ and $T>0$,  
\begin{align*}
    \P_{(0, T)}(Z_T=(n,0))
    =\lim_{t\to T, \delta\to 0}\P_{(0, T)}(|X_t-n|<\delta)
    =\frac{P_T(n)}{H_T(0)}.
\end{align*} 
\end{lemma}

\begin{proof}
Let $\varep_1,\varep_2>0$. By the joint continuity of $P_t(x)$, there exist $\delta_1,\delta_2>0$ such that $|x-n|<\delta_1$ and $T-\delta_2<t<T$ imply $|P_t(x)-P_T(n)|<\varep_1$ and $|\int_{|x|<\delta_1}P_{T-t}(x)\,dx -1|<\varep_2$. In addition, we may assume that $\delta_1$ is small enough such that $|\int_{|x|<\delta_1}P_{T-t}(x)\,dx -1|<\varep_2$ yields
\begin{align*}
    1-\varep_2
    <\int_{|x|<\delta_1}P_{T-t}(x)\,dx=\int_{|x-n|<\delta_1}P_{T-t}(x-n)\,dx
    \le \int_{|x-n|<\delta_1}H_{T-t}(x)\, dx<1
\end{align*}
and hence 
\[
    \left|\int_{|x-n|<\delta_1}H_{T-t}(x)\,dx -1\right|<\varep_2.
\] 
From this, we have
\begin{align*}
    \left|\P_{(0, T)}(|X_t-n|<\delta_1)-\frac{P_T(n)}{H_T(0)}\right|
    &=  \left|\int_{|x-n|<\delta_1}\frac{P_t(x)H_{T-t}(x)}{H_T(0)}\, dx-\frac{P_T(n)}{H_T(0)}\right|\\
    &\le \frac{1}{H_T(0)} \bigg( \int_{|x-n|<\delta_1}|P_t(x)-P_T(n)|H_{T-t}(x)\, dx\\
    &\qquad\qquad+P_T(n)\left|\int_{|x-n|<\delta_1}H_{T-t}(x)\, dx-1\right| \bigg)\\
    &\le \frac{\varep_1+P_T(n)\varep_2}{H_T(0)},
\end{align*}
for $|x-n|<\delta_1$ and $T-\delta_2<t<T$. So, we obtain
\[
    \P_{(0, T)}(X_T=n)
    =\lim_{t\to T, \delta\to 0}\P_{(0, T)}(|X_t-n|<\delta)
    =\frac{P_T(n)}{H_T(0)}.
    \qedhere
\]
\end{proof}
To obtain the probabilistic discrete second-order Riesz transforms $\cR^{(jk)}$,
we will let $T\to \infty$. For this purpose, the following observation will be useful. 
Since $H_T(0)\to 1$ and $(2\pi T)^{d/2}P_T(n)\to 1$, as $T\to\infty$, it  follows from the above lemma that for any integrable function $h$ on $\Z^d$,  
\begin{equation}\label{limNorm} 
    \lim_{T\to\infty} (2\pi T)^{d/2}\bE_{(0,T)}[h(X_T)] = \sum_{n\in \Z^d} h(n).  
\end{equation}

Let $f:\Z^d \to\R$ be compactly supported. Set 
\begin{align*}
    M_{t}=M^{f}_{t}=u_{f}(Z_{t}),  \quad \text{ for } t\in[0,T].
\end{align*}
By It\^o's formula, $M_t^f$  is martingale and 
\begin{align}\label{eq:def_M_t}
    M^{f}_{t}
    &= M^{f}_{0}+\int_{0}^{t\wedge T}\nabla u_{f}(X_{s},T-s)\cdot dX_{s} -\int_0^{t\wedge T} \frac{\partial u_f}{\partial s}(X_s,T-s)\, ds \\
    &\qquad\qquad +\frac{1}{2}\int_{0}^{t\wedge T}\Delta u_{f}(X_{s},T-s)\, ds \nonumber\\
    &= M^{f}_{0}+\int_{0}^{t\wedge T}\nabla u_{f}(Z_{s})\cdot dX_{s} -\int_{0}^{t\wedge T}\frac{\nabla H(Z_{s})\cdot \nabla u_{f}(Z_{s})}{H(Z_{s})}\, ds 
    \nonumber\\
    &= M^{f}_{0}+\int_{0}^{t\wedge T}\nabla u_{f}(Z_{s})\cdot dB_{s}. \nonumber 
\end{align}

Let $\fM_{d}(\R)$ be the space of all $d\times d$ real matrices and denote its norm by 
\[
    \|A\|=\sup\{|Av|:v\in\R^{d}, |v|\leq 1\}.
\]
By abuse of notation, for a matrix-valued function $A:\R^d\times \R_+\to\fM_{d}(\R)$, that is, $A(x,s)=(a_{ij}(x,s))_{1\leq i,j\leq d}$, $x\in\R^d$ and $s>0$, we define 
\begin{align*}
    \|A\| =\sup_{(x,s)\in\R^d\times\R_+}\|A(x,s)\| =\sup\{|A(x,s)v|:v\in\R^{d}, |v|\leq 1, (x,s)\in\R^d\times\R_+\}.
\end{align*}
 The martingale transform of $(M^f_{t})_{t\geq0}$ with respect to $A(x,s)$ is defined by
\begin{align*}
    A\star M^f_t
    &:= \int_{0}^{t\wedge T}A(Z_{s})\nabla u_f(Z_{s})\cdot dB_{s}\\
    &= \int_{0}^{t\wedge T}A(Z_{s})\nabla u_f(Z_{s})\cdot dZ_{s}-\int_{0}^{t\wedge T}\frac{A(Z_s)\nabla u_f(Z_{s}) \cdot \nabla H(Z_{s})}{H(Z_{s})}\, ds. 
\end{align*}
From the martingale inequalities in~\cite{Bur84}, their extensions in~\cite{Cho} and~\cite{BanOse} and the fact that $M_T^f=f_\ext(Z_T)$ and  $f_\ext(Z_T)=f(X_T)$ under the measure $\P_{(0,T)}$ by  Proposition \ref{boundary-values1} and Lemma \ref{boundary-values2}, we have 

\begin{theorem}\label{thm:MartTrans}
Let $A$ be a matrix-valued function with $\|A\|<\infty$. Then 
\begin{align}\label{Burkholder} 
    \bE_{(0, T)}| A\star M_T^f|^p& \leq (p^\ast -1)^p\|A\|^p\bE_{(0, T)} |M_T^f|^p\\
    &=(p^\ast -1)^p\|A\|^p\bE_{(0, T)}|f(X_T)|^p\nonumber\\
    &=(p^\ast -1)^p\|A\|^p\sum_{n\in \Z^d}|f(n)|^p\frac{P_T(n)}{H_T(0)}.\nonumber
\end{align}
Suppose that $b, B$ are constants such that $-\infty<b<B<\infty$ and $A$ has the property that for all $\xi\in \R^d$,
\begin{equation}\label{Non-Symm} 
    b|\xi|^2\leq A(\xi, t)\xi\cdot \xi \leq B|\xi|^2,\quad \text{for all} \,\, (x, t)\in\R^d\times\R_+.
\end{equation}
Then, 
\begin{align}\label{Choi1}
    \bE_{(0, T)}| A\star M_T^f|^p\leq C_{p,b,B}^p\|A\|^p \bE_{(0, T)}|M_T^f|^p=C_{p,b,B}^p\|A\|^p\sum_{n\in \Z^d}|f(n)|^p\frac{P_T(n)}{H_T(0)}, 
\end{align}
where $C_{p,b,B}$ is the Choi ~\cite{Cho} constant. In particular, when $b=0$ and $B=1$, $C_{p,0,1}=\gamma(p)$ is the constant in \eqref{RieszSharp2} satisfying the basic asymptotics in \eqref{ChoiAsym}.   
\end{theorem}

\subsection{Discrete projections of martingale transforms} 

Let $A:\R^d\times \R_+\to\fM_d(\R)$ be a matrix-valued function with $\|A\|<\infty$ and $f:\Z^d\to\R$ be compactly supported. For fixed $T>0$, we define
\[
    \cP^T_A(f)(n):=\bE_{(0,T)}[A\star M^f_T|X_T=n].
\]
From Theorem~\ref{thm:MartTrans} and the fact that the conditional expectation is a contraction in $L^p$ (with respect to the measure $\P_{(0,T)}$), we have the following

\begin{theorem}\label{thm:main1}
\begin{itemize}
\item[(i)] 
For any bounded matrix-valued function $A$, 
\begin{equation}\label{AnyA}
   \bE_{(0, T)}|\cP_A^Tf(X_T)|^p\leq (p^\ast -1)^p\|A\|^p\sum_{n\in \Z^d}|f(n)|^p\frac{P_T(n)}{H_T(0)}.
\end{equation} 
\item[(ii)] 
For any bounded matrix-valued function $A$ satisfying the additional property~\eqref{Non-Symm} in Theorem~\ref{thm:MartTrans},
\begin{equation}\label{AnyAChoy}
    \bE_{(0, T)}|\cP_A^Tf(X_T)|^p\leq C_{p,b,B}^p\|A\|^p\sum_{n\in \Z^d}|f(n)|^p\frac{P_T(n)}{H_T(0)},
\end{equation} 
In particular, if $b=0, B=1$, then $C_{p, b, 1}=\gamma(p)$ as in~\eqref{ChoiAsym}. 
\end{itemize}
\end{theorem} 

\subsection{The kernels of the projection operators}\label{ProjectionOperators}  

Our next goal is to show that as $T\uparrow \infty$, the operator $\cP_A^T$ converges weakly to an operator $\cP_A$ and to compute its kernel. Let $f\in\ell^p(\Z^d)$ and $g\in\ell^q(\Z^d)$ be compactly supported with $\frac{1}{p}+\frac{1}{q}=1$. By conditioning on $X_T$ and using the fact that the covariation of $\cP^{T}_A(f)(X_T)$ and $g(X_T)$ is given by 
\[
\langle\cP^{T}_A(f)(X_T), g(X_T)\rangle=\int_0^T \nabla u_g(Z_s)\cdot A(Z_s)\nabla u_f(Z_s)\, ds,
\]
we have 
\begin{align}
    \bE_{(0,T)}[\cP^{T}_A(f)(X_T)\,g(X_T)]
    &=\bE_{(0,T)}\left(\bE_{(0,T)}\Big[A\star M^f_T \,\Big|\, X_T\Big]\,g(X_T)\right)   \nonumber\\
    &=\bE_{(0,T)}\left(\bE_{(0,T)}\Big[A\star M^f_T\, g(X_T)\,\Big|\,X_T\Big]\right) \nonumber\\
    &=\bE_{(0,T)}\big[A\star M^f_T\,M_T^g\big]\nonumber\\
    &= \bE_{(0,T)}\left[\int_0^T \nabla u_g(Z_s)\cdot A(Z_s)\nabla u_f(Z_s)\, ds\right]\nonumber\\
    &=\int_0^T\bE_{(0,T)} \left[\nabla u_g(Z_s)\cdot A(Z_s)\nabla u_f(Z_s)\right]\, ds, \nonumber\\
    &= \int_0^T \int_{\R^d}\nabla u_g(x,T-s)\cdot A(x,T-s)\nabla u_f(x,T-s)\frac{P_{s}(x)H_{T-s}(x)}{H(0,T)}\,dx ds\nonumber\\
    &= \int_0^T \int_{\R^d}\nabla u_g(x,s)\cdot A(x,s)\nabla u_f(x,s)\frac{P_{T-s}(x)H_s(x)}{H_T(0)}\,dx ds.\label{eq:dualexpectation}
\end{align}

Thus, using \eqref{limNorm} we can define an operator $\cP_A$ by
\begin{align}\label{limitcP_A}
    &\sum_{n\in\Z^d} \cP_A(f)(n)g(n)\nonumber\\
    &= \lim_{T\to\infty}(2\pi T)^{d/2} \bE_{(0,T)}[\cP^{T}_A(f)(X_T) g(X_T)]\nonumber\\
    &= \lim_{T\to\infty}(2\pi T)^{d/2} \int_0^T \int_{\R^d}\nabla u_g(x,s)\cdot A(x,s)\nabla u_f(x,s)\frac{P_{T-s}(x)H_s(x)}{H_T(0)}\,dx ds, 
\end{align}
provided this limit exists. The proof of the existence of the limit and its identification is proved in Proposition~\ref{lim-TtoInf} using the heat kernel estimates  in Lemma~\ref{lem:intbound} together with the  following lemma.  

\begin{lemma}\label{lem:intbound}
For all $n,m\in\Z^d$, we have
\begin{align*}
    \int_0^\infty \int_{\R^d}\Big|\nabla \left(\frac{P_s(x-m)}{H_s(x)}\right)\Big|\Big|\nabla \left(\frac{P_s(x-n)}{H_s(x)}\right)\Big|H_s(x)\,dx ds\leq 1.
\end{align*}

\end{lemma}
\begin{proof}
{
Let $f,g\in\ell^2(\Z^d)$ be compactly supported. From~\eqref{eq:dualexpectation}, Cauchy-Schwarz inequality and Theorem~\ref{thm:main1} (with $p=2$) we have 
\begin{align*}
    &\Big|\int_0^T \int_{\R^d}\nabla u_g(x,s)\cdot A(x,s)\nabla u_f(x,s)\frac{P_{T-s}(x)H_s(x)}{H_T(0)}\,dx ds\Big|\\
    &= \Big|\bE_{(0,T)}[\cP^{T}_A(f)(X_T) g(X_T)]\Big|\\
    &\leq \|A\|\left(\bE_{(0,T)}|f(X_T)|^2)\right)^{1/2}\left(\bE_{(0,T)}|g(X_T)|^2)\right)^{1/2}. 
\end{align*}
Taking 
\[
    A(x, s)=\frac{\nabla{u_f(x, s)}\otimes{\nabla u_g(x, s)}}{|\nabla{u_f(x, s)}||\nabla{u_g(x, s)}|}
\]
gives
\begin{align}
    &\int_0^T \int_{\R^d}|\nabla u_g(x,s)||\nabla u_f(x,s)|\frac{P_{T-s}(x)H_s(x)}{H_T(0)}\,dx ds\\
    & \leq \left(\bE_{(0,T)}|f(X_T)|^2\right)^{1/2}\left(\bE_{(0,T)}|g(X_T)|^2\right)^{1/2}.\nonumber 
\end{align}
Recall that for $m\in\Z^d$, $\delta_m:\Z^d\to\R$ is defined by $\delta_m(l)=1$ if $l=m$ and otherwise 0. With  $f=\delta_m, g=\delta_n$, \eqref{eq:def_u} gives that  
\[ 
    u_f(x, s)=\left(\frac{P_s(x-m)}{H_s(x)}\right), \quad u_g(x,s)=\left(\frac{P_s(x-n)}{H_s(x)}\right).
\]
Thus, 
\begin{align*}
    &\int_0^T \int_{\R^d}\Big|\nabla \left(\frac{P_s(x-m)}{H_s(x)}\right)\Big|\Big|\nabla \left(\frac{P_s(x-n)}{H_s(x)}\right)\Big|\frac{{P_{T-s}(x)H_s(x)}}{H_T(0)}\,dx ds\\
    & \leq \left(\frac{P_T(m)}{H_T(0)}\right)^{1/2}\left(\frac{P_T(n)}{H_T(0)}\right)^{1/2},\nonumber
\end{align*} 
which, after multiplying by $H_t(0)$, is the same as 
\begin{align*}
    &\int_0^T \int_{\R^d}\Big|\nabla \left(\frac{P_s(x-m)}{H_s(x)}\right)\Big|\Big|\nabla \left(\frac{P_s(x-n)}{H_s(x)}\right)\Big|P_{T-s}(x)H_s(x)\,dx ds\\
    & \leq \left(P_T(m)\right)^{1/2}\left(P_T(n)\right)^{1/2}.\nonumber
\end{align*} 
Since
\begin{align*}
    \frac{P_{T-s}(x)}{\left(P_T(m)\right)^{1/2}\left(P_T(n)\right)^{1/2}}
    =\left(\frac{T}{T-s}\right)^{d/2}\exp\left(-\frac{|x|^2}{2(T-s)}+\frac{|n|^2+|m|^2}{4T}\right),
\end{align*}
for fixed $x\in\R^d$ and $s>0$, as $T\to\infty$ we have
\[
    \lim_{T\to\infty}\frac{P_{T-s}(x)\ind_{[0,T/2]}(s)}{\left(P_T(m)\right)^{1/2}\left(P_T(n)\right)^{1/2}}=1.
\]
By Fatou's lemma, we conclude
\begin{align*}
    &\int_0^{\infty} \int_{\R^d}
    \Big|\nabla \left(\frac{P_s(x-m)}{H_s(x)}\right)\Big|\Big|\nabla \left(\frac{P_s(x-n)}{H_s(x)}\right)\Big|H_s(x)\,dx ds\\
    &=\int_0^{\infty} \int_{\R^d}\lim_{T\to\infty}
    \Big|\nabla \left(\frac{P_s(x-m)}{H_s(x)}\right)\Big|\Big|\nabla \left(\frac{P_s(x-n)}{H_s(x)}\right)\Big|
    \frac{P_{T-s}(x)H_s(x)\ind_{[0,{T}/{2}]}(s)}{\left(P_T(m)\right)^{1/2}\left(P_T(n)\right)^{1/2}}\,dx ds\\
    &\le\liminf_{T\to\infty}\int_0^{T} \int_{\R^d}
    \Big|\nabla \left(\frac{P_s(x-m)}{H_s(x)}\right)\Big|\Big|\nabla \left(\frac{P_s(x-n)}{H_s(x)}\right)\Big|
    \frac{P_{T-s}(x)H_s(x)}{\left(P_T(m)\right)^{1/2}\left(P_T(n)\right)^{1/2}}\,dx ds\\
    &\leq 1,
\end{align*}
which concludes the proof of the Lemma. 

Note that the above argument can be applied for any $1<p<\infty$ (when we use Theorem~\ref{thm:main1}) leading to the bound $(p^*-1)$ which is optimal when $p=2$, hence our choice.} \qedhere

\end{proof}

\begin{proposition}\label{lim-TtoInf}
For a compactly supported function $f:\Z^d\to \R$, the limit in~\eqref{limitcP_A} exists and 
\[
    \cP_A(f)(n)=\sum_{m\in\Z^d}\cK_A(n,m)f(m),
\]
where 
\begin{align}\label{eq:kernel}
    \cK_A(n,m)= \int_0^\infty \int_{\R^d}H_s(x)\nabla \left(\frac{P_s(x-m)}{H_s(x)}\right)\cdot A(x,s)\nabla \left(\frac{P_s(x-n)}{H_s(x)}\right)\,dx ds.    
\end{align}
\end{proposition}

\begin{proof} 
 Let $m\in\Z^d$ and $g=\delta_m$. From the relation~\eqref{limitcP_A}, one can write $\cP_A$ as. 
\begin{align*}
    \cP_A(f)(m)
    &=\sum_{n\in\Z^d} \cP_A(f)(n)g(n) \\
    &= \lim_{T\to\infty}(2\pi T)^{d/2} \int_0^T \int_{\R^d}\nabla \left(\frac{P_s(x-m)}{H_s(x)}\right)\cdot A(x,s)\nabla u_f(x,s)\frac{P_{T-s}(x)H_s(x)}{H_T(0)}\,dx ds.
\end{align*}
Since a compactly supported function $f:\Z^d\to\R$ can be written as $f=\sum_{i=1}^j f_i \delta_{m_k}$ with $m_1,m_2,\ldots,m_j\in\Z^d$, and $f_1,\ldots,f_j\in\R$, and $\lim_{T\to\infty} H_T(x)=1$, by Lemma~\ref{lem:hlimit}, it suffices to show that
\begin{align*}
    \lim_{T\to\infty}(2\pi T)^{d/2} &\int_0^T \int_{\R^d}\nabla \left(\frac{P_s(x-m)}{H_s(x)}\right)\cdot A(x,s)\nabla \left(\frac{P_s(x-n)}{H_s(x)}\right){P_{T-s}(x)H_s(x)}\,dx ds\\
    &=\int_0^\infty \int_{\R^d}H_s(x)\nabla \left(\frac{P_s(x-m)}{H_s(x)}\right)\cdot A(x,s)\nabla \left(\frac{P_s(x-n)}{H_s(x)}\right)\,dx ds.
\end{align*}
Set 
\begin{align*}
    J(x,s,T)&:=(2\pi T)^{d/2}\nabla \left(\frac{P_s(x-m)}{H_s(x)}\right)\cdot A(x,s)\nabla \left(\frac{P_s(x-n)}{H_s(x)}\right){P_{T-s}(x)H_s(x)},\\
    I(T) &:= \int_0^\infty \int_{\R^d} J(x,s,T)\,dx ds = I_1(T)+I_2(T),\\
    I_1(T)&:= \int_0^{\infty}\int_{\R^d} J(x,s,T)\ind_{\{T/2<s\le T\}}\, dx ds,\\
    I_2(T)&:= \int_0^{\infty} \int_{\R^d} J(x,s,T)\ind_{\{0< s\le T/2\}}\, dx ds.
\end{align*}

Our goal is to show the existence of the limit of $I(T)$ as $T\to\infty$. In fact, we will show that
\begin{align*}
    \lim_{T\to\infty}I(T)
    &=\lim_{T\to\infty}I_2(T)\\
    &=\int_0^\infty \int_{\R^d}H_s(x)\nabla \left(\frac{P_s(x-m)}{H_s(x)}\right)\cdot A(x,s)\nabla \left(\frac{P_s(x-n)}{H_s(x)}\right)\,dx ds.
\end{align*}
Since $T$ is large enough, it is natural to assume that $T>2$. 

{\bf\noindent Case 1:} We first claim that $\lim_{T\to\infty}I_1(T)=0$. It follows from Lemma~\ref{lem:der-h-bound} that
\begin{align*}
    &\left|\nabla\left(\frac{P_s(x-m)}{H_s(x)}\right)\cdot A(x,s)\nabla \left(\frac{P_s(x-n)}{H_s(x)}\right)\right|\\
    &\le \frac{P_s(x-m)P_s(x-n)}{H^2_s(x)}\|A\| \left|\frac{\nabla P_s(x-m)}{P_s(x-m)}-\frac{\nabla H_s(x)}{H_s(x)} \right|\left|\frac{\nabla P_s(x-n)}{P_s(x-n)}-\frac{\nabla H_s(x)}{H_s(x)}\right|\\
  &\le C_d\frac{P_s(x-m)P_s(x-n)}{s H^2_s(x)}\|A\| \left(\frac{|x-m|}{\sqrt{s}}+C_d\right)\left(\frac{|x-n|}{\sqrt{s}}+C_d\right).
\end{align*}
By~\eqref{eq:h-prod} and~\eqref{eq:thetalowerbound}, we have, with $s>T/2>1$,  
\[
H_s(x)\ge \left(\frac{e^{2\pi^2}-3}{e^{2\pi^2}-1}\right)^d = C_d>0.
\]
Since 
\[
    \left(\frac{|x-m|}{\sqrt{s}}+C_d\right) \left(\frac{|x-n|}{\sqrt{s}}+C_d\right) \le  \left(\frac{|x|}{\sqrt{s}}+C\right)^2,
\]
we have
\begin{align*}
    \left|\nabla\left(\frac{P_s(x-m)}{H_s(x)}\right)\cdot A(x,s)\nabla \left(\frac{P_s(x-n)}{H_s(x)}\right)\right|  
    \le\frac{ C_d\|A\|}{sH_s(x)}P_s(x-m)P_s(x-n)\left(\frac{|x|}{\sqrt{s}}+C\right)^2.
\end{align*}
Here and below, we use $C$ to denote constants depending on $d, m, n$.
Using $|x-m|^2+|x-n|^2\ge  |x|^2 - |m|^2- |n|^2$, we get
\begin{align*}
    &P_s(x-m) P_s(x-n) P_{T-s}(x)\\
    &\leq C_d (T-s)^{-\tfrac{d}{2}}s^{-d}\exp\left(  -\frac{1}{2s}\left(|x|^2 - |m|^2- |n|^2\right) -\frac{|x|^2}{2(T-s)} \right)\\
    &\le C(T-s)^{-\tfrac{d}{2}}s^{-d}\exp\left(-\frac{|x|^2}{2u} \right)
\end{align*}
where $u=s(T-s)/T$. Note that $u\le s$ for $T/2\le s\le T$.
Combining these, we have
\begin{align*}
    |I_1(T)|
    &\le C\|A\| T^{d/2}\int_{T/2}^{T} \int_{\R^d} \frac{1}{s}P_s(x-m)P_s(x-n)P_{T-s}(x) \left(\frac{|x|}{\sqrt{s}}+C\right)^2 \, dxds\\
    &\le C\|A\| \int_{T/2}^{T} s^{-\frac{d}{2}-1}u^{-\frac{d}{2}}\int_{\R^d} e^{-\frac{|x|^2}{2u}} \left(\frac{|x|}{\sqrt{u}}+C\right)^2 \, dxds.
\end{align*}
By change of variable $y=x/\sqrt{u}$, 
\begin{align*}
    |I_1(T)|
    &\le C \|A\|\int_{T/2}^{T}s^{-\frac{d}{2}-1} \int_{\R^d} e^{-\frac{|y|^2}{2}} \left(|y|+C\right)^2 \, dyds\\
    &= C\|A\| \int_{T/2}^{T}s^{-\frac{d}{2}-1} ds\\
    &=C\|A\|\left((T/2)^{-\frac{d}{2}}-T^{-\frac{d}{2}}\right)\\
   & =C\|A\| T^{-d/2}, 
\end{align*}
which gives $\lim_{T\to\infty}I_1(T)=0$ as desired.

{\bf\noindent Case 2:} 
Note that $(2\pi T)^{d/2}P_{T-s}(x)$ is uniformly bounded in $x\in\R^d$, $0< s\le T/2$, and $T\ge 1$.  Furthermore, for fixed $x$ and $s$, the limit as $T\to\infty$ is equal to 1. We have
\begin{align*}
    |J(x,s,T)\ind_{\{0< s\le T/2\}}|
    &\le C\|A\| \Big|\nabla \left(\frac{P_s(x-m)}{H_s(x)}\right)\Big|\Big|\nabla \left(\frac{P_s(x-n)}{H_s(x)}\right)\Big|H_s(x)
\end{align*}
and the right hand side is integrable by Lemma~\ref{lem:intbound}. Therefore, the dominated convergence theorem yields 
\[
    \lim_{T\to\infty} I_2(T)= \int_0^\infty \int_{\R^d}H_s(x)\nabla \left(\frac{P_s(x-m)}{H_s(x)}\right)\cdot A(x,s)\nabla \left(\frac{P_s(x-n)}{H_s(x)}\right)\,dx ds. \qedhere
\]

\end{proof}

As observed above,
\begin{align*}
    \bE_{(0,T)}[|f(X_T)|^p] 
    =\sum_{n\in\Z^d}|f(n)|^p\frac{P_{T}(n)}{H_T(0)}.
\end{align*}
Since $H_T(0)\to 1$ and $(2\pi T)^{d/2}P_T(n)\to 1$ as $T\to\infty$, we have
\[
    \lim_{T\to\infty} (2\pi T)^{d/2}\bE_{(0,T)}[|f(X_T)|^p] = \|f\|_{\ell^p}^p.  
\]

From this, Proposition~\ref{lim-TtoInf} and Theorem~\ref{thm:main1}, we have 

\begin{theorem}\label{limCor}
\begin{enumerate}[label=(\roman*), leftmargin=*]
\item 
For any bounded matrix-valued function $A$, 
\begin{equation}\label{AnyACor}
    \|\cP_Af\|_{\ell^p}\leq \|A\| (p^\ast -1)\|f\|_{\ell^p}.
\end{equation} 
\item 
For any bounded matrix-valued function $A$ satisfying the additional property~\eqref{Non-Symm} in Theorem~\ref{thm:MartTrans},
\begin{equation}\label{AnyAChoyCor}
    \|\cP_Af\|_{\ell^p}\leq \|A\| C_{p, b, B}\|f\|_{\ell^p}.
\end{equation} 
In particular, if $b=0, B=1$, then $C_{p, 0, 1}=\gamma(p)$ as in~\eqref{ChoiAsym}. 
\end{enumerate}
\end{theorem}

\section{Probabilistic second-order discrete Riesz transforms}\label{ProbDisRiesz} 

For the rest of the paper, we assume that $d\geq 2$.  Fix $j,k=1,2,\ldots, d$. Let $A^{jk}$ be the matrix whose $a^{jk}=1$ and otherwise 0. Motivated by the construction for the classical Riesz transforms $R^{(jk)}$ in \cite[Proposition 2.2]{BanMen}, we define the \emph{probabilistic second-order discrete Riesz transforms} by $\cR^{(jk)}=\cP_{A^{jk}}$ where $A^{jk}=(a_{r,s}^{jk})$ is the $d\times d$ matrix with 
\begin{align}\label{Projmat}
    a_{j, k}^{jk}=-1\quad \text{and} \quad a_{r, s}^{jk}=0,\,\,  \text{if}\,\,  r\neq j \,\, \text{or}\,\,  s\neq k. 
\end{align}
From~\eqref{eq:kernel}, we see that $\cR^{(jk)}(f)(n) =\sum_{m\in\Z^d}\cK^{(jk)}(n-m)f(m)$ where
\begin{align}\label{ProKer}
    \cK^{(jk)}(m)
    =-\int_0^{\infty}\int_{\R^d}H_s(x)\frac{\partial}{\partial x_j} \left(\frac{P_s(x)}{H_s(x)}\right)\frac{\partial}{\partial x_k} \left(\frac{P_s(x-m)}{H_s(x)}\right)\,dx ds.
\end{align}
When $j\neq k$, it follows that $\cR^{(jk)}$ is also given by $-\cR_{A^{kj}}$.  That is, $2\cR^{(jk)}=\cP_{A}$ where $A=A^{jk}-A^{kj}$. Note that $\cK_A(n,n)=0$. Since $\|A\|\leq 1$, we have the following corollary from Theorem~\ref{limCor} (exactly as in~\cite[Theorem 2.2]{BanMen}).

\begin{corollary}\label{HalfSharp} 
For all $j,k=1,2,\ldots,d$, with $j\ne k$, we have
\begin{align}
    \|2\cR^{(jk)}\|_{\ell^p\to\ell^p}\leq (p^\ast -1),\,\,\,\, 
    \|\cR^{(kk)}-\cR^{(jj)}\|_{\ell^p\to\ell^p}\leq (p^\ast -1),
\end{align} 
and when $k=j$, we have 
\begin{align} 
    \|\cR^{(jj)}\|_{\ell^p\to\ell^p}\leq \gamma(p). 
\end{align} 
\end{corollary} 

Before presenting the proof of Theorem~\ref{thm:main}, we highlight the relationship between the probabilistic discrete second-order Riesz transforms $\cR^{(jk)}$ with kernels  $\cK^{(jk)}(m)$  given by~\eqref{ProKer} and the discrete Calder\'on-Zygmund discrete second-order Riesz transforms $R^{(jk)}_{\dis}$ with kernels $K^{(jk)}_{\dis}$ given by~\eqref{disRiesz2}. 

For notational convenience, we use $H(x, t)$ in place of $H_t(x)$. 

\begin{theorem}\label{prop:Kdecomp} 
For all $j,k=1, 2, \ldots, d$ and all $m\in \Z^d\setminus\{0\}$, we have 
\begin{align}\label{EqProp:Kdecomp}
    \cK^{(jk)}(m) =  U(m)K^{(jk)}_{\dis}(m),
\end{align}
where
\begin{align*}
    U(m)=  -\frac{1}{\pi^{d/2}\Gamma((d+2)/2)}\int_0^\infty\int_{\R^d} \frac{s^{\tfrac{d}{2}}e^{-s-|z|^2}}{H\left(\tfrac{z|m|}{2\sqrt{s}}+\tfrac{m}{2}, \tfrac{|m|^2}{4s}\right)} \, dzds.
\end{align*}
\end{theorem}

\begin{proof} 
Our goal is to show that the ratio $\cK^{(jk)}(m) /K^{(jk)}_{\dis}(m)$ is equal to $U(m)$ as defined above. Note that $\cK^{(jk)}(m)$ is defined in~\eqref{ProKer} and $K^{(jk)}_{\dis}(m)$ is given in~\eqref{disRiesz2}. 

First, we simplify the integrand of $\cK^{(jk)}(m)$ in~\eqref{ProKer}
\begin{align}\label{eq:cKintegrand}
    H_s(x)\frac{\partial}{\partial x_j} \left(\frac{P_s(x)}{H_s(x)}\right)\frac{\partial}{\partial x_k} \left(\frac{P_s(x-m)}{H_s(x)}\right)
\end{align}
using the representation of $H_s(x)$ in terms of $\theta$ in~\eqref{eq:h-prod}.

Recall that the derivative of $\theta(z, t)$ with respect to the first variable is denoted by $\partial_z\theta(z,t)$.
From 
\begin{align*}
    \frac{\partial }{\partial x_j} P_t(x-n) = -\frac{(x_j-n_j)}{t}P_t(x-n),\quad
    \frac{\partial }{\partial x_j} H(x,t) = H(x,t)\partial_z\theta(x_j,t)/\theta(x_j,t),
\end{align*}
we get
\begin{align*}
    &H(x,t)\frac{\partial}{\partial x_j} \left(\frac{P_t(x)}{H(x,t)}\right)\frac{\partial}{\partial x_k} \left(\frac{P_t(x-m)}{H(x,t)}\right)\\
    &= \frac{1}{H(x,t)^3}\left(H\partial_j P_t(x) - P_t(x)\partial_j H(x,t)\right)\left(H(x,t)\partial_k P_t(x-m) - P_t(x-m)\partial_k H(x,t)\right)\\
    &= \frac{\partial_j P_t(x) \partial_k P_t(x-m)}{H(x, t)} - \frac{P_t(x)\partial_j H(x, t) \partial_k P_t(x-m)}{H(x, t)^2} \\
    &\qquad \qquad- \frac{P_t(x-m)\partial_j P_t(x) \partial_k H(x, t)}{H(x, t)^2} +\frac{P_t(x) P_t(x-m) \partial_j H(x, t) \partial_k H(x, t)}{H(x, t)^3}\\
    &= \frac{P_t(x) P_t(x-m)}{H(x, t)}\left(\frac{x_j}{t}+\frac{\partial_{z} \theta(x_j,t)}{\theta(x_j,t)}\right)\left(\frac{x_k-m_k}{t}+\frac{\partial_{z} \theta(x_k,t)}{\theta(x_k,t)}\right)\\
    &= \prod_{l=1}^d\frac{P^{(1)}_t(x_l)P^{(1)}_t(x_l-m_l)}{\theta(x_l,t)}\left(\frac{x_j}{t}+\frac{\partial_{z} \theta(x_j,t)}{\theta(x_j,t)}\right)\left(\frac{x_k-m_k}{t}+\frac{\partial_{z} \theta(x_k,t)}{\theta(x_k,t)}\right).
\end{align*}

Now, the integral of~\eqref{eq:cKintegrand} over $\R^d$ can be split into the product of integrals over $\R$ as
\begin{align*}
    \int_{\R^d}H(x,t)\frac{\partial}{\partial x_j} \left(\frac{P_t(x)}{H(x,t)}\right)\frac{\partial}{\partial x_k} \left(\frac{P_t(x-m)}{H(x,t)}\right)\,dx =F_2(m_j, t)F_3(m_k, t)\prod_{\substack{i=1,\\i\neq j,k}}^d F_1(m_i, t),
\end{align*}
where
\begin{align*}
    F_1(m_i,t) &=
    \int_{\R}\frac{P^{(1)}_t(x_i)P^{(1)}_t(x_i-m_i)}{\theta(x_i,t)}\, dx_i,\\
    F_2(m_j,t) &=\int_{\R}\frac{P^{(1)}_t(x_j)P^{(1)}_t(x_j-m_j)}{\theta(x_j,t)}\left(\frac{x_j}{t} +\frac{\partial_{z}\theta(x_j,t)}{\theta(x_j,t)}\right)\, dx_j,\\
    F_3(m_k,t) &=\int_{\R}\frac{P^{(1)}_t(x_k)P^{(1)}_t(x_k-m_k)}{\theta(x_k,t)}\left(\frac{x_k-m_k}{t} +\frac{\partial_{z}\theta(x_k,t)}{\theta(x_k,t)}\right)\, dx_k.
\end{align*}
We will first simplify $F_2(m_j,t)$ and $F_3(m_k,t)$. By the change of variable $y=m_k-x_k$ and  the symmetry of $P^{(1)}_t(x_k)$ in $x_k$, we have
\begin{align*}
    \int_{\R}\frac{P^{(1)}_t(x_k)P^{(1)}_t(x_k-m_k)}{\theta(x_k,t)} \, dx_k
    &=-\int_{\R}\frac{P^{(1)}_t(m_k-y)P^{(1)}_t(-y)}{\theta(m_k-y,t)} \, dy\\
    &=-\int_{\R}\frac{P^{(1)}_t(y-m_k)P^{(1)}_t(y)}{\theta(m_k-y,t)} \, dy.
\end{align*}
Since $\theta(m_k-y,t)=\theta(y,t)$ by~\eqref{eq:theta}, the integral is zero and so 
\[
    \int_{\R}\frac{(x_k-m_k)P^{(1)}_t(x_k)P^{(1)}_t(x_k-m_k)}{t\theta(x_k,t)} \, dx_k
    =\int_{\R}\frac{x_k P^{(1)}_t(x_k)P^{(1)}_t(x_k-m_k)}{t\theta(x_k,t)} \, dx_k.
\]
For $l=j,k$, the change of variable $y=m_l/2 -x$ yields
\begin{align*}
    \int_{\R}{P^{(1)}_t(x)P^{(1)}_t(x-m_l)}\frac{\partial_{z}\theta(x,t)}{\theta(x,t)^2}\, dx
    =-\int_{\R}{P^{(1)}_t\left(\frac{m_l}{2}-y\right)P^{(1)}_t\left(\frac{m_l}{2}+y\right)}\frac{\partial_{z}\theta(\frac{m_l}{2}-y,t)}{\theta(\frac{m_l}{2}-y,t)^2}\, dy.
\end{align*}
Since $P^{(1)}_t\left(\frac{m_l}{2}-y\right)P^{(1)}_t\left(\frac{m_l}{2}+y\right)$ and $\theta(\frac{m_l}{2}-y,t)$ are even functions in $y$, the integrand is odd. So, we have 
\[
\int_{\R}{P^{(1)}_t(x)P^{(1)}_t(x-m_l)}\frac{\partial_{z}\theta(x,t)}{\theta(x,t)^2}\, dx=0.
\]
Therefore, we get
\begin{align*}
    F_2(m_i,t) =
    F_3(m_i,t) =
    \int_{\R}\frac{x P^{(1)}_t(x)P^{(1)}_t(x-m_i)}{t\theta(x,t)}\, dx,\qquad i=j,k.
\end{align*}
Note that the change of variable $z\mapsto z+m_i/2$ yields that for $i\ne j,k$, 
\begin{align*}
    F_1(m_i,t)
    =\int_{\R}\frac{P^{(1)}_t(z)P^{(1)}_t(z-m_i)}{\theta(z,t)}\, dz 
    =\frac{e^{-m_i^2/(4t)}}{2\pi t}\int_{\R}\frac{e^{-z^2/t}}{\theta(z+m_i/2,t)}\,dz,
\end{align*}
and for $i= j,k$,
\begin{align*}
    F_2(m_i,t)
    =F_3(m_i,t)
    =\int_{\R}\frac{z P^{(1)}_t(z)P^{(1)}_t(z-m_i)}{t \theta(z,t)}\, dz 
    =\frac{m_i e^{-m_i^2/(4t)}}{4\pi t^2}\int_{\R}\frac{e^{-z^2/t}}{\theta(z+m_i/2,t)}\,dz.
\end{align*}
Therefore, $\cK^{(jk)}(m)$ can be written as
\begin{align*}
    \cK^{(jk)}(m) 
    &=-\int_0^\infty F_2(m_j, t)F_3(m_k, t)\prod_{\substack{i=1,\\i\neq j,k}}^d F_1(m_i, t)\, dt\\
    &=-\frac{m_j m_k}{2^{d+2}\pi^{d}}
        \int_0^\infty t^{-d-2}e^{-\frac{|m|^2}{4t}}\int_{\R^d}\frac{e^{-|x|^2/t}}{H\left(x+\tfrac{m}{2},t\right)}\,dxdt.
\end{align*}
By the change of variables $x=z\sqrt{t}$ and $t=\frac{|m|^2}{4s}$, we finally get
\begin{align*}
    &\cK^{(jk)}(m) \\
    &=-\frac{m_j m_k}{2^{d+2}\pi^{d/2}}\int_0^\infty t^{-\frac{d}{2}-2}e^{-\frac{|m|^2}{4t}} \left(\pi^{-\tfrac{d}{2}}\int_{\R^d}\frac{e^{-|z|^2}}{H(z\sqrt{t}+\tfrac{m}{2},t)}\,dz\right)dt\nonumber\\
    &=-\frac{m_j m_k}{2^{d+2}\pi^{d/2}}\int_0^\infty \left(\frac{|m|^2}{4s}\right)^{-\frac{d}{2}-2}e^{-s} \left(\pi^{-\tfrac{d}{2}}\int_{\R^d}\frac{e^{-|z|^2}}{H\left(\tfrac{z|m|}{2\sqrt{s}}+\tfrac{m}{2}, \tfrac{|m|^2}{4s}\right)}\,dz\right)\frac{|m|^2}{4s^2}ds\nonumber\\
    &=-c_d\frac{m_j m_k}{|m|^{d+2}} \left( \frac{1}{\pi^{d/2} \Gamma\left(\tfrac{d+2}{2}\right)} \int_0^\infty\int_{\R^d} \frac{s^{\tfrac{d}{2}}e^{-s-|z|^2}}{H\left(\tfrac{z|m|}{2\sqrt{s}}+\tfrac{m}{2}, \tfrac{|m|^2}{4s}\right)} \, dzds  \right),\\
\end{align*}
where the constant $c_d$ is as in~\eqref{disRiesz2}. Comparing with the definition of $K^{(jk)}_{\dis}$ (see~\eqref{disRiesz1}) finishes the proof.  \qedhere

\end{proof}

Below, we give some corollaries that provide some further information and properties of the function $U(m)$. Recall that the density of the inverse Gamma distribution $S$ with parameters $\alpha$ and $\beta$ is given by  
$f(t; \alpha, \beta)=\frac{\beta^{\alpha}}{\Gamma(\alpha)}\frac{e^{-\frac{\beta}{t}}}{t^{\alpha+1}}$,  for $t>0$.

\begin{corollary}\label{InverseGammaU}
Let $S$ be a random variable with an inverse Gamma distribution of parameters $\alpha$ and $\beta$ given by $\alpha=\frac{d+2}{2}$, $\beta=\frac{|m|^2}{4}$. Let $B_t$ be the standard $d$-dimensional Brownian motion. Then, for $m\in\Z^d$, with $m\ne 0$,
\begin{align}
    U(m)=-\bE\left(\frac{1}{H_{S}\left(\frac{B_S}{\sqrt{2}}+\frac{m}{2}\right)}\right) \quad \text{and}\quad U(m)=U(-m). 
\end{align}  
\end{corollary} 

\begin{proof} 
As in the proof of Theorem \ref{prop:Kdecomp}, we have 
\begin{align*}
    \cK^{(jk)}(m)
    &=-\frac{m_j m_k}{2^{d+2}\pi^{d/2}}\int_0^\infty t^{-\frac{d}{2}-2}e^{-\frac{|m|^2}{4t}} (\pi t)^{-d/2}\int_{\R^d}\frac{e^{-|z|^2/t}}{H_t(z+m/2)}\,dzdt\\
    &=-\frac{m_j m_k}{2^{d+2}\pi^{d/2}}\int_0^\infty t^{-\frac{d}{2}-2}e^{-\frac{|m|^2}{4t}} \left((2\pi t)^{-d/2}\int_{\R^d}\frac{e^{-\tfrac{|z|^2}{2t}}}{H_t\left(\tfrac{z}{\sqrt{2}}+\tfrac{m}{2}\right)}\,dz\right)dt\\
    &=-c_d \frac{m_j m_k}{2^{d+2}\Gamma\left(\tfrac{d+2}{2}\right)}
    \int_0^\infty 
    t^{-\tfrac{d}{2}-2}e^{-\frac{|m|^2}{4t}} 
    \bE\left[\frac{1}{H_t\left(\frac{B_t}{\sqrt{2}}+\frac{m}{2}\right)}\right]
    dt\\
    &=-c_d\frac{m_j m_k}{|m|^{d+2}}\int_0^{\infty}f(t; \alpha, \beta) \bE\left[\frac{1}{H_t\left(\frac{B_t}{\sqrt{2}}+\frac{m}{2}\right)}\right] dt\\
    &= -K^{(jk)}_{\dis}(m)
    \bE\left[\frac{1}{H_S\left(\frac{B_S}{\sqrt{2}}+\frac{m}{2}\right)}\right], 
\end{align*}
and the corollary follows.
\end{proof} 

\begin{corollary}\label{U-new}
$U(m)$ is a function of the number of even coordinates in $m=(m_1,\ldots,m_d)$, which we denote by $e(m)$,  and $|m|$.
\end{corollary} 

\begin{proof} 
Define 
\[
    q_j(t) := \int_{-1/2}^{1/2}\frac{\theta(x,t/2)}{\theta(x+j/2,t)}\, dx,\qquad j=0,1.
\]
Recall again that $Q=[-\tfrac{1}{2}, \tfrac{1}{2})^d$.
Since $H_t(x)$ is periodic in $x$, we get
\begin{align*}
    (\pi t)^{-d/2}\int_{\R^d}\frac{e^{-|z|^2/t}}{H_t(z+m/2)}\,dz
    &=\sum_{n\in\Z^d }(\pi t)^{-d/2}\int_{Q+n}\frac{e^{-|z|^2/t}}{H_{t}(z+m/2)}\,dz\\
    &=\int_{Q}\frac{H_{\frac{t}{2}}(z)}{H_t(z+m/2)}\,dz\\
    &=q_0(t)^{e(m)} q_1(t)^{d-e(m)}=:V_m(t).
\end{align*}
By Theorem~\ref{prop:Kdecomp}, $U(m)$ can be written as 
\begin{align*}
    U(m)
    &= -\frac{|m|^{d+2}}{2^{d+2}\Gamma(\frac{d+2}{2})}\int_0^\infty t^{-\frac{d}{2}-2}e^{-\frac{|m|^2}{4t}} V_{m}(t)\, dt\\
    &= -\frac{1}{\Gamma(\frac{d+2}{2})}\int_0^\infty t^{\frac{d}{2}}e^{-t} V_{m}\left(\frac{|m|^2}{4t}\right)\, dt\\
    &=- \frac{1}{\Gamma(\frac{d+2}{2})}\int_0^\infty t^{\frac{d}{2}}e^{-t} q_0\left(\frac{|m|^2}{4t}\right)^{e(m)} q_1\left(\frac{|m|^2}{4t}\right)^{d-e(m)}\, dt.   \qedhere
\end{align*} \end{proof} 

By the definition of the Gamma function, we see
\begin{align*}
    U(m)+1 
    =\frac{1}{\Gamma(\frac{d+2}{2})}\int_0^\infty t^{\frac{d}{2}}e^{-t} \left( 1-q_0\left(\frac{|m|^2}{4t}\right)^{e(m)} q_1\left(\frac{|m|^2}{4t}\right)^{d-e(m)}\right)\, dt.
\end{align*}
One can see from Figure~\ref{fig:q} that $q_0(t)$ and $q_1(t)$ are close to 1 when $t$ is away from 0, which implies that $U(m)$ is quite close to $-1$ for large $|m|$. Since replacing $U(m)$ with $-1$ in the kernel of $\cK^{(jk)}$ (see~\eqref{EqProp:Kdecomp}) gives the kernel of the discrete second-order Riesz transform, up to sign, it is interesting to measure how $U(m)$ differs from $-1$.

\begin{figure}[t]
\includegraphics[scale=0.65]{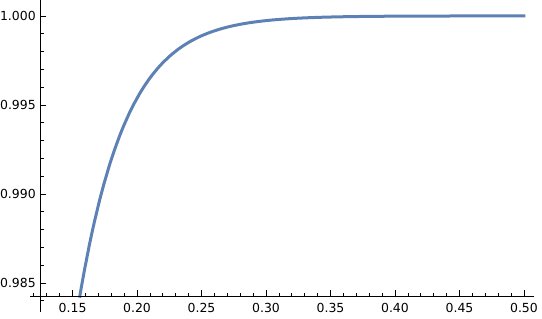}
\includegraphics[scale=0.65]{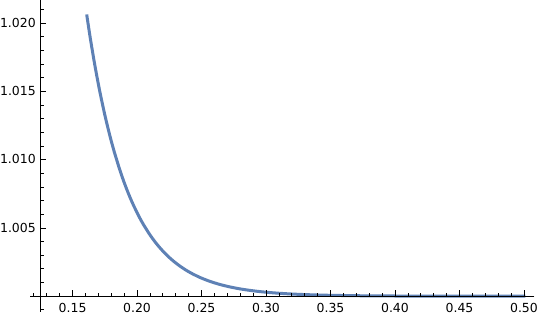}
\caption{The graphs of $q_0(t)$ (Left) and $q_1(t)$ (Right) for $1/8\le t\le 1/2$}
\label{fig:q}
\end{figure}

\begin{corollary}\label{lem:Ulimit} 
There exist dimensional constants $C_1,C_2, C_3, C_4>0$ such that 
\begin{enumerate}[label=(\roman*), leftmargin=*]
\item
$|U(m)+1|\le \frac{C_2}{|m|}$ for $|m|>C_1$, and
\item 
$0<C_3\leq |U(m)|\leq C_4$  for all $m\in\Z^d\setminus\{0\}$.  
\end{enumerate}
\end{corollary}

\begin{proof}
It follows from Theorem~\ref{prop:Kdecomp} that
\begin{align*}
    &|U(m)+1| \\
    &\le  \frac{1}{\pi^{d/2}\Gamma((d+2)/2)}\int_0^\infty\int_{\R^d} s^{\tfrac{d}{2}}e^{-s-|z|^2}\left|\frac{1}{H_{\frac{|m|^2}{4s}}\left(\tfrac{z|m|}{2\sqrt{s}}+\tfrac{m}{2}\right)}-1\right| \, dzds\\
    &=C_d \int_0^{\tfrac{|m|^2}{4}}\int_{\R^d} s^{\tfrac{d}{2}}e^{-s-|z|^2}\left|\frac{1}{H_{\frac{|m|^2}{4s}}\left(\tfrac{z|m|}{2\sqrt{s}}+\tfrac{m}{2}\right)}-1\right| \, dzds \\
    &\qquad\qquad+C_d \int_{\tfrac{|m|^2}{4}}^{\infty}\int_{\R^d} s^{\tfrac{d}{2}}e^{-s-|z|^2}\left|\frac{1}{H_{\frac{|m|^2}{4s}}\left(\tfrac{z|m|}{2\sqrt{s}}+\tfrac{m}{2}\right)}-1\right| \, dzds.
\end{align*}
By Lemma~\ref{lem:hlimit}, there exists a dimensional constant $C_d$  such that
\begin{align*}
    &|U(m)+1| \\
    &\le \frac{C_d}{|m|}\Bigg( \int_0^{\tfrac{|m|^2}{4}}\int_{\R^d} s^{\tfrac{d+1}{2}}e^{-s-|z|^2} \, dzds
    + \int_{\tfrac{|m|^2}{4}}^\infty\int_{\R^d} s^{\tfrac{d+1}{2}}e^{-s-|z|^2+\frac{C_2}{|m|^2}s} \, dzds\Bigg)\\
    &=\frac{C_d}{|m|}\left(\int_0^{\tfrac{|m|^2}{4}}s^{\tfrac{d+1}{2}}e^{-s} ds+\int_{\tfrac{|m|^2}{4}}^\infty s^{\tfrac{d+1}{2}}e^{-s+\frac{C_2}{|m|^2}s} ds\right).
\end{align*}
Note that
\[
    \int_0^{\tfrac{|m|^2}{4}}s^{\tfrac{d+1}{2}}e^{-s} ds
    \le \int_0^{\infty}s^{\tfrac{d+1}{2}}e^{-s} ds =\Gamma\left(\frac{d+3}{2}\right).
\]
Assuming $|m|^2\ge 2C_2$, we have 
\begin{align*}
    \int_{\tfrac{|m|^2}{4}}^\infty s^{\tfrac{d+1}{2}}e^{-s+\frac{C_2}{|m|^2}s}\, ds
    \le 
    C_d\int_{\tfrac{|m|^2}{4}}^\infty s^{\tfrac{d+1}{2}}e^{-\frac{s}{2}}\, ds
    \le C_d2^{\frac{d+3}{2}}\Gamma\left(\frac{d+3}{2}\right).
\end{align*}
Combining these, we obtain the bound in (i).

By (i), there exists $C_d$ such that $|U(m)+1|\le \tfrac14$ for $|m|\ge C_d$, which gives $\frac{1}{2}\leq |U(m)|$. On the other hand, 
\[ 
    |U(m)|=\bE\left[\frac{1}{H_S\left(\frac{B_S}{\sqrt{2}}+\frac{m}{2}\right)}\right] 
\]
is strictly positive for every $m$ and hence so is  $\min_{1\leq |m|\leq C_d}|U(m)|$, which gives the lower bound in (ii) for all $m$.

It follows from (i) that $|U(m)|\leq \frac{C_2}{|m|}+1$, for all $|m|>2C_1$. By~\eqref{EqProp:Kdecomp}, for $|m|\leq 2C_1$, we have
\[
    |U(m)|\le \max_{1\le |n|\le 2C_1}\left|\frac{\cK^{(jk)}(n)}{K^{(jk)}_{\dis}(n)}\right|<\infty. \qedhere
\] 
\end{proof}

\begin{remark}
Note that the results in Corollary~\ref{lem:Ulimit} remain valid if we replace $m\in\Z^d$ with $z\in\R^d$, $|z|>1$ because $\frac{\cK^{(jk)}(z)}{K^{(jk)}_{\dis}(z)}$ is well-defined and continuous.
\end{remark}

Motivated by Theorem \ref{prop:Kdecomp}, we set  
\begin{align}\label{convolution1}
\cJ^{(jk)}(m)=c_d \frac{m_j m_k}{|m|^{d+2}}(U(m)+1)\ind_{\Z^d\setminus\{0\}}(m)
\end{align} 
so that  by \eqref{disRiesz2}, 
\begin{align}\label{convolution2}
\cJ^{(jk)}(m)=K^{(jk)}_{\dis}(m)(U(m)+1), 
\end{align}
where $K^{(jk)}_{\dis}$ are the kernels for the discrete second-order Riesz transforms. 
With this, we see that  
\begin{align}\label{convolution2}
(\cR^{(jk)}+R_{\dis}^{(jk)})f(n)=\cJ^{(jk)}* f(n).  
\end{align} 

\begin{corollary}\label{KerJ} 
For all $j, k=1, 2, \dots, d$, we have
\begin{align}\label{L1KerJ}
    \|\cJ^{(jk)}\|_{\ell^1(\Z^d)}\leq C_d,
\end{align}
where $C_d$ dependents on $d$. 
\end{corollary} 

\begin{proof} 
It follows from Corollary~\ref{lem:Ulimit} that   
\begin{align}\label{Referee1} 
    |\cJ^{(jk)}(m)| 
    = c_d \frac{|m_j m_k|}{|m|^{d+2}}|U(m)+1|
    \le \begin{cases} 
    \frac{C_d}{|m|^{d}}, \quad |m|\leq C_d,\\
    \frac{C_d}{|m|^{d+1}}, \quad |m|> C_d, 
\end{cases} 
m\ne 0.  
\end{align}
The conclusion follows. 
\end{proof}

\begin{corollary}\label{J-kern}
For all $j, k=1, 2, \dots, d$ with $j\ne k$, we have
\begin{align}
    \|2R_{\dis}^{(jk)}\|_{\ell^p\to\ell^p}&\leq (p^*-1)+\|2\cJ^{(jk)}\|_{\ell^1(\Z^d)},\\
    \|R_{\dis}^{(jj)}-R_{\dis}^{(kk)}\|_{\ell^p\to\ell^p}&\leq (p^*-1)+\|\cJ^{(jj)}-\cJ^{(kk)}\|_{\ell^1(\Z^d)},\\
    \|R_{\dis}^{(jj)}\|_{\ell^p\to\ell^p}&\leq \gamma(p) +  \|\cJ^{(jj)}\|_{\ell^1(\Z^d)}, 
\end{align} 
where $\gamma(p)$ is the Choi constant. 
\end{corollary} 

\section{Proof of Theorems~\ref{thm:main} and~\ref{thm:diag}}\label{sec:proof}

The proof of Theorem~\ref{thm:main} follows from the next two propositions,~\eqref{RieszSharp1_1},~\eqref{RieszSharp1_2}, and Corollary~\ref{HalfSharp}.   The following proposition is proved in~\cite[Theorem 6.2]{BanKimKwa} for singular integrals with kernels of the form  $K(x)=\frac{\Omega(x)}{|x|^{d}}$ with $\Omega(x)$ odd on the sphere. Although  the function $\Omega$ defining the classical second-order Reisz transform $R^{(jk)}$ are not odd, the proof can be modified in the present case for $j\ne k$ using the fact that the dilated version of the discrete operators $R_{\dis}^{(jk)}$ and $R_{\dis}^{(jj)}-R_{\dis}^{(kk)}$ approximate the classical second-order Riesz transforms $R^{(jk)}$ and $R^{(jj)}-R^{(kk)}$, respectively. Our argument to justify this approximation relies on the fact that the kernel $K_{\dis}^{(jk)}$ in the $j$-th variable is odd when we fix the other variables. However, if $j=k$, then $K_{\dis}^{(jj)}$ is symmetric in the $j$-th variable, so that our proof does not provide the inequality for $R^{(jj)}$. 

Note that, from~\eqref{FullSharp},~\eqref{FullSharp2} and Corollary~\ref{HalfSharp}, we already have for all $j\neq k$, 
\begin{align*}
    \|2\cR^{(jk)}\|_{\ell^p\to\ell^p}\leq (p^*-1) &= \|2R^{(jk)}\|_{L^p\to L^p},\\
    \|\cR^{(jj)}-\cR^{(kk)}\|_{\ell^p\to\ell^p}\leq (p^*-1)&= \|R^{(jj)}-R^{(kk)}\|_{L^p}. 
\end{align*}

\begin{proposition}\label{discont1}
For $j,k=1,2,\cdots,d$,   $j\neq k$, we have 
\begin{align}
    \|R^{(jk)}\|_{L^p\to L^p}
    &\le \|R_{\dis}^{(jk)}\|_{\ell^p\to\ell^p}\label{upper1},\\
    \|R^{(jj)}-R^{(kk)}\|_{L^p\to L^p}
    &\le \|R_{\dis}^{(jj)}-R_{\dis}^{(kk)}\|_{\ell^p\to\ell^p}\label{upper2}.
\end{align}
\end{proposition}

\begin{proof} 
Let $F$ be a smooth function in $L^p(\R^d)$ with compact support. For $\varep>0$, define the $p$-norm preserving dilation $\tau_\varep F(x): = \varep^{d/p}F(\varep x)$. That is,  $\|\tau_{\varep}F\|_p = \|F\|_p$. We also define the continuous-discrete Riesz transforms by

\begin{align}\label{ContDiscR}
    \wt{R}_{\dis}^{(jk)}(F)(x)=\sum_{m\in\Z^d}K^{(jk)}_{\dis}(m)F(x-m).
\end{align}
To simplify the notation, set  
\[
    R=R^{(jk)},\qquad  
    R_{\dis}=R_{\dis}^{(jk)}, \qquad
    \wt{R}_{\dis}=\wt{R}_{\dis}^{(jk)}.   
\]
It is well-known (see Remark \ref{Dis-ContNorms} below) that 
\begin{align}\label{Condis=dis}
    \|R_{\dis}\|_{\ell^p\to\ell^p}=\|\wt{R}_{\dis}\|_{L^p\to L^p}.
\end{align} 
We claim that 
\begin{align}
    \|R(F)\|_{L^p(\R^d)} \leq \|\wt{R}_{\dis}(F)\|_{L^p(\R^d)}.
\end{align} 
From \eqref{Condis=dis} and the claim, \eqref{discont1} follows.

For $x=(x_1,x_2,\cdots,x_d)\in\R^d$, we use the notations $\wh{x}=(x_1,\cdots,x_{j-1},x_{j+1},\cdots,x_d)$, $x = (x_j, \wh{x})$, and similarly for $m\in\Z^d$. Using the fact 
\begin{align}\label{eq:Kodd}
K^{(jk)}_{\dis}(m_j, \wh{m})=-K^{(jk)}_{\dis}(-m_j, \wh{m}),\qquad
K^{(jk)}_{\dis}(m_j, \wh{m})=-K^{(jk)}_{\dis}(m_j, -\wh{m})
\end{align}
for $m\in\Z^d\setminus\{0\}$, we can write
\begin{align}\label{eq:Riemannsum}   
    \wt{R}^\varep_{\dis}(F)(x)
    &:= \tau_{1/\varep}\wt{R}_{\dis}\tau_{\varep}(F)(x)\nonumber\\
    &= \sum_{m\in\Z^d} K^{(jk)}_{\dis}(m)F(x-\varep m)\nonumber\\
    &= \varep^{d}\sum_{m\in\Z^d} K^{(jk)}_{\dis}(\varep m)F(x-\varep m)\nonumber\\
    &= -\frac{\varep^{d}}{4}\sum_{m\in\Z^d} K^{(jk)}_{\dis}(\varep m)
    \{
    F(x_j+\varep m_j, \wh{x}+\varep \wh{m})
    -F(x_j-\varep m_j, \wh{x}+\varep \wh{m})\\
    &\qquad -F(x_j+\varep m_j, \wh{x}-\varep \wh{m})
    +F(x_j-\varep m_j, \wh{x}-\varep \wh{m})\}.\nonumber
\end{align}

Since $F$ is smooth with compact support, there exists $C>0$ such that
\begin{align}\label{eq:Festim}
    &|F(x_j+\varep m_j, \wh{x}+\varep \wh{m})
    -F(x_j-\varep m_j, \wh{x}+\varep \wh{m})
    -F(x_j+\varep m_j, \wh{x}-\varep \wh{m})
    +F(x_j-\varep m_j, \wh{x}-\varep \wh{m})|\\
    &\qquad\le 2\varep |m_j||\partial_j F(x_j-\varep m_j, \wh{x}+\varep \wh{m})-\partial_j F(x_j-\varep m_j, \wh{x}-\varep \wh{m})|+2\varep^2 |m_j|^2\max_x |\partial_{jj}F(x)|\nonumber\\
    &\qquad\le 4\varep^2 |m_j||\wh{m}|\max_x|\nabla_{\wh{x}}\partial_j F(x)|+C\varep^2 |m_j|^2 \nonumber\\
    &\qquad\le C \varep^2 |m|^2.\nonumber
\end{align}
With this estimate at hand, let $r>0$ and split $\wt{R}^\varep_{\dis}(F)(x)$ into the following two sums: 
\begin{align*}
    \wt{R}^\varep_{\dis}(F)(x) = \sum_{|\varep m|>r}K^{(jk)}_{\dis}(m)F(x-\varep m)+\sum_{|\varep m|\le r}K^{(jk)}_{\dis}(m)F(x-\varep m) =: I + II.
\end{align*}
Since $I$ is a Riemann sum, from~\eqref{eq:Riemannsum} we get
\begin{align*}
    \lim_{\varep\to 0}I 
    &= -\frac{c_d}{4}\int_{|y|>r} \frac{y_j y_k}{|y|^{d+2}} \{F(x_j+y_j, \wh{x}+\wh{y}) -(F(x_j-y_j, \wh{x}+\wh{y})\\
    &\qquad-F(x_j+y_j, \wh{x}-\wh{y}) +(F(x_j-y_j, \wh{x}-\wh{y})\}\, dy\\
    &= c_d\int_{|y|>r} \frac{y_j y_k}{|y|^{d+2}}F(x-y)\, dy.
\end{align*}
On the other hand, it follows from~\eqref{eq:Riemannsum},~\eqref{eq:Festim}, and the definition of $K^{(jk)}_{\dis}$ that
\begin{align*}
    |II| \le C \varep^d \sum_{\{|\varep m|\le r, m\neq 0\}}|\varep m|^{2-d} \le Cr^2.
\end{align*}
Letting $r\to 0 $, we have
\begin{align}\label{Rapproximation}
    \lim_{\varep\to 0}\wt{R}^\varep_{\dis}(F)(x) = R(F)(x).
\end{align}
Using~\eqref{Rapproximation} and Fatou's lemma, we conclude that 
\begin{align*}
   \|R(F)\|_{L^p(\R^d)} \leq \liminf_{\varep\to 0}\|\wt{R}^\varep_{\dis}(F)\|_{L^p(\R^d)}=
   \|\wt{R}_{\dis}(F)\|_{L^p(\R^d)},
\end{align*}
where the last equality comes from the fact that the dilation by $\varep$ preserves the $L^p$-norm.  This proves the claim and~\eqref{upper1} follows. 

For the inequality~\eqref{upper2}, we may assume without loss of generality that $j=1$ and $k=2$. By change of variables, the kernel for $R_{\dis}^{(11)}-R_{\dis}^{(22)}$ can be written as
\[
    K^{(11)}_{\dis}(m)-K^{(22)}_{\dis}(m)
    =\frac{c_d}{|m|^{d+2}}(m_1^2-m_2^2) 
    =2c_d \frac{u_1 u_2}{|u|^{d+2}} 
\]
where $u_1=\frac{1}{\sqrt{2}}({m_1+m_2})$, $u_2=\frac{1}{\sqrt{2}}({m_1-m_2})$, and $u_j=m_j$ for all $j\ge 3$. Thus, the previous argument leads to \eqref{upper2} and this concludes the proof of the Proposition.
\end{proof}

\begin{remark}\label{rmk:whyjnek}
Note that the proof of Proposition~\ref{discont1} uses the fact that the kernel $K^{(jk)}_{\dis}$ is odd in both $m_j$ and $\wh{m}$ for $m=(m_j,\wh{m})\in\Z^d\setminus\{0\}$ and $j\ne k$ (see~\eqref{eq:Kodd}). Since these properties  are not valid for the case $j=k$, a different method is needed for $R^{(jj)}_{\dis}$. 
\end{remark}

\begin{remark}\label{Dis-ContNorms}
A version of~\eqref{Condis=dis} for the discrete Hilbert transform goes back to Riesz~\cite{Riesz}, and it holds for quite general kernels. For this, we refer the reader to~\cite{Laeng, Pierce} and~\cite[Theorem 6.2]{BanKimKwa}.
\end{remark} 

The proof of Theorem~\ref{thm:main} is completed with the following 

\begin{proposition}\label{lem:S-pnorm}
For $j,k=1,2,\ldots,d$ with $j\neq k$,  
\begin{align}
    \|\cR^{(jk)}\|_{\ell^p\to\ell^p}&\geq \|R^{(jk)}\|_{L^p \to L^p}\label{SharpS1},\\
    \|\cR^{(jj)}-\cR^{(kk)}\|_{\ell^p\to\ell^p}&\geq\|R^{(jj)}-R^{(kk)}\|_{L^p \to L^p }\label{SharpS2}.
\end{align}
\end{proposition}

\begin{proof} 
As in Proposition \ref{discont1}, for a smooth function $F\in L^p(\R^d)$ with compact support, we define the continuous-discrete operator by 
\begin{align}\label{ContDiscR}
    \wt{\cR}^{(jk)}(F)(x)=\sum_{m\in\Z^d}\cK^{(jk)}(m)F(x-m)
\end{align}
and the dilated operator 
\[
(\wt{\cR}^{(jk)})^{\varep}(F)(x) := \tau_{1/\varep}\wt{\cR}^{(jk)}\tau_{\varep}(F)(x).
\]
We claim that 
\begin{align}
    \lim_{\varep\to 0}(\wt{\cR}^{(jk)})^{\epsilon}(F)(x) &= R^{(jk)}(F)(x)\label{claim1},\\
    \lim_{\varep\to 0}\left((\wt{\cR}^{(jj)})^{\epsilon} (F)(x) -(\wt{\cR}^{(kk)})^{\epsilon} (F)(x)\right)&= R^{(jj)}(F)(x)-R^{(kk)}(F)(x).\label{claim2}
\end{align}
As before, the inequalities~\eqref{SharpS1} and~\eqref{SharpS2} follow from~\eqref{claim1},~\eqref{claim2} and Fatou's lemma as in the proof of Proposition~\ref{discont1}.

We will first prove~\eqref{claim1}. Using the change of variables, one can reduce the second assertion to~\eqref{claim1}. To simplify notation again, we fix $j,k$ and drop the superscript ${(jk)}$.  Let $x\in\R^d$ be fixed. By~\eqref{convolution2},  $\cR=-R_{\dis}+\cJ$ and by Proposition~\ref{discont1}, $\wt{R}_{\dis}^{\varep}(F)\to R(F)$ as $\varep\to 0$. Thus, we need to show that the limit of
\begin{align*}
\wt{\cJ}^\varep (F)(x)
&=\varep^{-d/p}\sum_{m\ne 0}K_{\dis}(m)(U(m)+1)(\tau_\varep F)(x/\varep- m)\\
&=\sum_{|m|\ge 1}K_{\dis}(m)(U(m)+1)F(x-\varep m)\\
&=\sum_{|m|\ge 1}K_{\dis}(m)(U(m)+1)F(x+\varep m)
\end{align*}
is equal to 0, as $\varep\to 0$. Here, we used the fact that $K_{\dis}(m)=K_{\dis}(-m)$ and $U(m)=U(-m)$, by Corollary \ref{InverseGammaU}.
Since $F$ has compact support, there exists $R>0$ such that $F(x+y)=0$ for $|y|\ge R$.  Let $\delta>0$. By Corollary~\ref{lem:Ulimit}, there exists $r>1$ such that $|U(m)+1|<\delta$ for $|m|\ge r$.  Choose $\varep>0$ small enough that $r<R/\varep$. Since $F(x+\varep m)=0$ for $|m|\ge R/\varep$, we have
\begin{align*}
    |\wt{\cJ}^\varep (F)(x)|
    &\le \bigg|\sum_{1\le|m|<R/\varep}K_{\dis}(m)(U(m)+1)F(x+\varep m)\bigg|\\
    &\le  \bigg| \sum_{r\le |m|<R/\varep}K_{\dis}(m)(U(m)+1)F(x+\varep m)\bigg|\\
    &\qquad\qquad+\bigg|\sum_{1\le |m|<r}K_{\dis}(m)(U(m)+1)F(x+\varep m)\bigg|.
\end{align*}
Using
\begin{align*}
    &|F(x_i+\varep m_i, \wh{x}+\varep \wh{m})
    -F(x_i-\varep m_i, \wh{x}+\varep \wh{m})\\
    &\qquad\qquad-F(x_i+\varep m_i, \wh{x}-\varep \wh{m})
    +F(x_i-\varep m_i, \wh{x}-\varep \wh{m})|\\
    &\qquad\le C\varep^2 |m|^2,
\end{align*}
as in the proof of Proposition~\ref{discont1} (see~\eqref{eq:Kodd}), we get 
\begin{align*}
    |\wt{\cJ}^\varep (F)(x)|
    &\le  C\varep^2\left(\sum_{r\le |m|<R/\varep}|K_{\dis}(m)||U(m)+1||m|^2+\sum_{1\le |m|<r}|K_{\dis}(m)||U(m)+1||m|^2\right).
\end{align*}
Since $|K_{\dis}(m)|\le C_d$ for $1\le |m|<r$ and $|K_{\dis}(m)|\le C_d|m|^{-d}$ for $r\le |m|<R/\varep$,
\begin{align*}
    |\wt{\cJ}^\varep (F)(x)|
    &\le  C\varep^2\left(\sum_{r\le |m|<R/\varep}|U(m)+1||m|^{2-d}+\sum_{1\le |m|<r}|U(m)+1||m|^2\right).
\end{align*}
It then follows from $|U(m)+1|<\delta$ for $r\le |m|<R/\varep$ and $|U(m)+1|\le |U(m)|+1\le C_d$ for $1\le |m|<r$ that
\begin{align*}
    |\wt{\cJ}^\varep (F)(x)|
    &\le  C\varep^2\left(\delta\sum_{r\le |m|<R/\varep}|m|^{2-d}+\sum_{1\le |m|<r}|m|^2\right)\\
    &\le  C\varep^2\left(C_{R,d}\delta \varep^{-2}+C_{r,d}\right)
    = C(\delta + \varep^2).
\end{align*}
Letting $\varep\to 0$, we get $\lim_{\varep\to0}|\wt{\cJ}^\varep (F)(x)|\le C_{d,R}\delta$. Then, we let $\delta\to 0$ to conclude the proof.

For the second assertion~\eqref{claim2}, we use the change of variables $u_1=\frac{1}{\sqrt{2}}({m_1+m_2})$, $u_2=\frac{1}{\sqrt{2}}({m_1-m_2})$, and $u_j=m_j$ for all $j\ge 3$. Then, one can apply the first assertion to conclude the result. 
\end{proof}

\begin{proof}[Proof of Theorem~\ref{thm:diag}]
It follows from Theorem~\ref{prop:Kdecomp} and Proposition~\ref{lem:Ulimit} that $ C_1 K^{(jj)}_{\dis}(m)\le K_{jj}(m) \le C_2 K^{(jj)}_{\dis}(m)$ for some $C_1>0$ and $C_2>1$. Since $K^{(jj)}_{\dis}(m)>0$ for all $m$, we conclude that $ C_3\|R^{(jj)}_{\dis}\|_{\ell^p\to\ell^p}\le \|\cR_{jj}\|_{\ell^p\to\ell^p}\le C_4\|R^{(jj)}_{\dis}\|_{\ell^p\to\ell^p}$.
\end{proof}

We end this section with some remarks concerning Fourier multipliers. Recall that for $f:\Z^d\to\R$, its Fourier transform is defined by 
\begin{align*}
    \cF(f)(\xi) = \sum_{n\in\Z^d}f(n)e^{-2\pi i n\cdot \xi},  \quad  \xi\in Q=[-\tfrac12, \tfrac12)^d,   
\end{align*}
with the inverse  Fourier transform of a periodic  function $f$ on $Q$ given by  
\[
    \cF^{-1}(f)(n)=\int_{Q}f(\xi)e^{2\pi i n\cdot \xi}\, d\xi.
\]
The following is a well-known result, see~\cite[Chapter VII]{SteWei} and~\cite[Proposition 3.2]{PieSte1}.    

\begin{proposition}
Suppose that $T$ is a convolution operator (translation invariant) of the form 
\[
    Tf(n)=\sum_{m\in \Z^d}f(m)K(m-n).
\]
Then $T$ is bounded on $\ell^2(\Z^d)$ if and only if there exists $m\in L^{\infty}(Q)$ such 
\[
    \cF(Tf)(\xi)=m(\xi)\cF(f)(\xi),\quad \xi\in Q.
\]
Furthermore, 
\[
    \|T\|_{\ell^2\to \ell^2}=\|m\|_{L^{\infty}}.
\]
\end{proposition} 

This proposition applies to the discrete Riesz $R^{(jk)}_{\dis}$ with convolution kernels $K^{(jk)}_{\dis}$ and to the probabilistic versions $\cR^{(jk)}$ with convolutions with the kernels $\cK^{(jk)}(m)$.  Thus for $j\ne k$, 
\begin{align}\label{Fourier1}
    \Big\|\sum_{m\in\Z^d}\cK^{(jk)}(m)e^{-2\pi i m\cdot \xi}\Big\|_{L^{\infty}}=  
    \Big\|\sum_{m\in\Z^d}c_d\frac{m_j m_k}{|m|^{d+2}}U(m)e^{-2\pi i m\cdot \xi}\Big\|_{L^{\infty}}=
    \frac{1}{2} 
\end{align} 
and for $j=k$, 
\begin{align}\label{Fourier2}
    \Big\|\sum_{m\in\Z^d}\cK^{(jj)}(m)e^{-2\pi i m\cdot \xi}\Big\|_{L^{\infty}}=    
    \Big\|\sum_{m\in\Z^d}c_d\frac{m_j^2}{|m|^{d+2}}U(m)e^{-2\pi i m\cdot \xi}\Big\|_{L^{\infty}}
    \leq 1,
\end{align} 
by the first equality on~\eqref{off-dia1} of Theorem~\ref{thm:main} and inequality in Theorem~\ref{thm:diag}, respectively. The Conjecture~\ref{BigConj} for $p=2$ ($j\neq k$) would follow by showing that removing the factor $U(m)$ does not affect the $L^{\infty}$-norms above. 
Since the function $U(m)$ converges rapidly to $-1$, it suffices to focus on the low-frequency regime where $|m|$ is small. In this regime, one can exploit the fact that $U(m)$ depends only on $|m|$ and $e(m)$, the number of even coordinates in $m \in \mathbb{Z}^d$. Combined with the antisymmetry of the kernel $m_j m_k |m|^{-d-2}$, this structure may allow one to show that the contribution from low-frequency terms is negligible. However, a rigorous justification of this approximation remains elusive.

\section{The discrete Beurling-Ahlfors transform}\label{sec:BA}

The Beurling-Ahlfors transform is the singular integral operator on $L^p(\bC)$ defined in the principal value sense by 
\begin{align*}
    Bf(z) = -\frac{1}{\pi}\int_{\bC}\frac{f(w)}{(z-w)^2} dw.
\end{align*}
Here we identify $\R^2$ with the complex plane $\bC$ and $dw$ is the Lebesgue measure. This operator has been extensively studied in the literature in large part due to its connections to other areas of analysis, such as regularity theory for quasiconformal mappings,  partial differential equations, and the well-known 1982 conjecture of T.~Iwaniec~\cite{Iwa82}. The latter asserts that the operator norm satisfies  $\|B\|_{p\to p}\leq (p^*-1)$. See for example,~\cite{Astala1, Ban, MR3018958, IwaMar, BanJan, Lehto, MR3558516, Volberg1, Dragi}, and the many references contained therein. Regarding Iwanienc's conjecture, the following estimate is known: 
\begin{align}\label{1.5bound}
    (p^*-1)\leq \|B\|_{p\to p}\leq 1.575(p^*-1). 
\end{align}
The lower bound was proved in~\cite{Lehto} and the upper bound in~\cite{BanJan}.  (We remind the reader that $1<p<\infty$ and $(p^*-1)$ is as in \eqref{BirkholdrConstant}.)

In~\cite[page 138]{CZ}, the discrete version of $B$ is discussed as a case of ``special interest'' where the authors say that it  ``can be considered as the simplest generalization of the Hilbert-Toeplitz linear form to space $E^2$'' ($\R^2$ in our case). 

Let $\Z_{\bC}=\{n+im:n,m\in\Z\}$ and $f:\Z_{\bC}\to\bC$. As in \eqref{DiscCalZyg}, define the discrete Beurling-Ahlfors transform $B_{\dis}$ on $\ell^p(\Z_{\bC})$ by
\begin{align*}
    B_{\dis}f(z) = -\frac{1}{\pi}\sum_{w\in{\Z_{\bC}\setminus\{0\}}}\frac{f(w)}{(z-w)^2},\quad z=n+im\in \Z_{\bC}.
\end{align*}
As before, by  H\"older's inequality for any $f\in \ell^p(\bC)$, $1<p<\infty$, the sum is absolutely convergent. 

If $z=x+iy$, then the kernel for the continuous Beurling-Ahlfors transform can be written as
\[
    K_B(z)=K_B(x,y)=\frac{-1}{\pi z^2}=\frac{1}{\pi}\frac{y^2-x^2+2ixy}{(x^2+y^2)^2}=\Omega(x,y)/(x^2+y^2)
\]
where
\[
    \Omega(x,y)=\frac{1}{\pi}\frac{y^2-x^2+2ixy}{x^2+y^2}.
\]
Thus, with our notation for classical second-order Riesz transforms on $\R^2$,
\[
    B= (R^{(22)}-R^{(11)})+2i R^{(21)}
\]
and 
\[
    B_{\dis}= (R_{\dis}^{(22)}-R_{\dis}^{(11)})+2i R_{\dis}^{(21)}.
\]

Note that $K_B$ is $C^1(\bC\setminus\{0\})$, $|K_B(z)|\le C_1|z|^{-2}$, and $|\nabla K_B(z)|\le C_1|z|^{-3}$. 
Thus, the kernel satisfies the conditions of Propositions 5.1 and 5.2 in~\cite{BanKimKwa}, which together with \eqref{1.5bound} give  
\begin{align}\label{FromCZ}
    (p^*-1)&\leq  \|B\|_{L^p\to L^p}\le \|B_{\dis}\|_{\ell^p\to \ell^p}\\
   & \le \|B\|_{L^p\to L^p}+C\leq 1.575 (p^*-1)+C, \nonumber 
\end{align} 
where $C$ is a universal constant independent of $p$. 
We define the  {\it  probabilistic discrete Beurling-Ahlfors}  transform  by 
\begin{align*}
    \cB
    =(\cR^{(22)}-\cR^{11})+2i\cR^{(21)}. 
\end{align*} 
By \eqref{Projmat} this is the projection $\cR_A$  of a martingale transform by the 
$2\times 2$ matrix 
\begin{align*}
    A=\begin{pmatrix}
        1 & -i \\
       -i & -1
    \end{pmatrix}.
\end{align*}
An easy calculation shows that the matrix $A$ has norm 2 when acting on vectors with complex coordinates and norm $\sqrt{2}$ when acting on vectors with real coordinates. From Theorem~\ref{limCor},  we have  
\begin{align}
    \|\cB f\|_{\ell^p(\Z_{\bC})}\leq 
    \begin{cases}
    2(p^*-1), & f: \Z_{\bC}\to\bC, \\
    \sqrt{2}(p^*-1), & f: \Z_{\bC} \to \bR.
    \end{cases}
\end{align}
Together with  $\lim_{\varep\to 0}\wt{\cB}^\varep (F)(x)=BF(x)$ and Fatou's Lemma as in the proof of Proposition~\ref{discont1}, we obtain 

\begin{theorem}\label{MainDisBA}
For all $1<p<\infty$, 
\begin{equation*} 
    (p^*-1)\leq \|B\|_{L^p\to L^p}\leq \|\cB\|_{\ell^p\to\ell^p}\leq 2(p^*-1). 
\end{equation*}
\end{theorem}

\begin{remark}
For $2\leq p<\infty$, applying the conformal martingale inequalities in~\cite{BanJan}, see~\cite[pp 837-839]{Ban} and~\cite[\S 11]{Jana} in place of Burkholder's inequality in Theorems~\ref{thm:MartTrans},  gives 
\begin{align}\label{Conformal}
    \|\cB\|_{\ell^p\to\ell^p}\leq \sqrt{2p(p-1)}. 
\end{align} 
This is the better bound used in~\cite{BanJan} to obtain the $1.575(p^*-1)$ bound in~\eqref{1.5bound}. However, in addition to the bound $\sqrt{2p(p-1)}$, the proof in~\cite{BanJan} uses the fact that $\|B\|_{L^2\to L^2}=1$ and an interpolation argument. Unfortunately, here we do not yet have $\|\cB\|_{\ell^2\to\ell^2}=1$. 

Finally, as with the Riesz transforms, the operator $B_{\dis}$ differs from the probabilistic operator $\cB$  by a convolution with the $\ell^1$-function (with $\Z_{\bC}$ identified with $\Z^2$)
\begin{align}\label{RefereeBA} 
    \cJ_B(m) 
 &= \frac{1}{\pi}\left(\frac{m_2^2}{|m|^{4}}-\frac{m_1^2}{|m|^{4}}+2i\frac{m_2 m_1}{|m|^{4}}\right)\left(U(m)+1\right)\\
 &=\frac{1}{\pi}\left(\frac{(m_2+im_1)^2}{|m|^4}\right)\left(U(m)+1\right)\nonumber
\end{align}
which as in \eqref{convolution1} has the bound
\[
    |\cJ_B(m)| \le 
    \begin{cases} 
        \frac{C}{|m|^{2}}, \quad |m|\leq C,\\
        \frac{C}{|m|^{3}}, \quad |m|> C, 
    \end{cases}
\] 
for some constant $C$.
From this, it follows that $\|B_{\dis}\|_{\ell^p\to \ell^p}\leq 2(p^*-1)+C$, without appealing to Calder\'on-Zygmund theory. 
\end{remark}
  
As in the case of Conjecture \ref{BigConj} for Riesz transforms, we have the natural conjecture for the Beurling-Ahlfors operators. 
\begin{conjecture}
\begin{equation}
    \|B\|_{L^p\to L^p}= \|\cB\|_{\ell^p\to\ell^p}=\|B_{\dis}\|_{\ell^p\to\ell^p}=(p^*-1). 
\end{equation} 
\end{conjecture}

As in the case of the Riesz transforms $R^{(ik)}_{\dis}$, verifying that $\|B_{\dis}\|_{\ell^2\to\ell^2}=1$ would already be interesting. Again, perhaps calculations similar to those in \cite{Gra1} to compute the $\ell^2$-norm of $H_{\dis}$ would yield the result. 

\section{Continuous Calder\'on-Zygmund operators}\label{CZ-Prob}

In this section, we address the following natural 

\begin{question}\label{ConCZ} 
    Do the probabilistic second-order discrete Riesz transforms  $\cR^{(jk)}$ arise as discrete analogues of Calder\'on-Zygmund singular integrals in the sense of~\eqref{DiscCalZyg} and~\eqref{ConCalZyg1}? 
\end{question}

We give an affirmative answer to this question. In fact, the kernels for these operators are given by replacing the lattice point $m\in \Z^d$ in the discrete operators $\cR^{(jk)}$ constructed above by the continuous point $z\in \R^d$, for $|x|\geq 1$, and adding the kernel of the classical second-order Riesz transform for $|x|<1$. The choice for $|x|<1$ is made so that, as in the discrete case, the resulting probabilistic continuous operators differ from the classical Riesz transforms by a convolution with an $L^1(\R^d)$-function. More precisely, we have

\begin{theorem}\label{thm:contiCZ}
Define the kernels
\begin{align}\label{ContinousKernels} 
   & \mathbf{K}^{(jk)}(x) = \cK^{(jk)}(x)\ind_{\{|x|\ge 1\}} - K^{(jk)}(x)\ind_{\{|x|<1\}}\\
    &=-\left(\int_0^{\infty}\int_{\R^d}H_s(z)\frac{\partial}{\partial z_j} \left(\frac{P_s(z)}{H_s(z)}\right)\frac{\partial}{\partial z_k} \left(\frac{P_s(z-x)}{H_s(z)}\right)\,dz\,ds\right) \ind_{\{|x|\ge 1\}}\nonumber\\
    &-c_d \frac{x_j x_k}{|x|^{d+2}}\ind_{\{|x|<1\}}\nonumber
\end{align} 
These give Calder\'on-Zygmund operators $\mathbf{R}^{(jk)}:L^p(\R^d)\to L^p(\R^d)$ with discrete analogues $\cR^{(jk)}$. That is, 
\[
    \mathbf{R}^{(jk)}(f)(x) = \mathbf{K}^{(jk)}\ast f(x),
\]
where $\mathbf{K}^{(jk)}$ are Calder\'on-Zygmund kernels satisfying \eqref{ConCalZyg1} and 
\[
    \cR^{(jk)}(f)(n) = \sum_{m} \cK_{\dis}^{(jk)}(m)f(n-m),\quad 
    \cK_{\dis}^{(jk)}(m)=\mathbf{K}^{(jk)}(m)\ind_{\Z^d\setminus\{0\}}(m).
\]
\end{theorem}

We call the $\mathbf{R}^{(jk)}$ probabilistic continuous second-order  Riesz transforms. Before proving Theorem~\ref{thm:contiCZ}, we show that, as in the case of the discrete operators, even without verifying that they are Calder\'on-Zygmund kernels, these operators map $L^p(\R^d)$ to $L^p(\R^d)$ and are weak-type $(1, 1)$.  Indeed, with $U$ as in Theorem~\ref{prop:Kdecomp}, we have  
\begin{align}
    \mathbf{K}^{(jk)}(x) 
    &= \cK^{(jk)}(x)\ind_{\{|x|\ge 1\}} - K^{(jk)}(x)\ind_{\{|x|<1\}}\\
    &= -c_d \frac{x_j x_k}{|x|^{d+2}}+c_d \frac{x_j x_k}{|x|^{d+2}}(U(z)+1)\ind_{\{|x|\ge 1\}}\nonumber\\
    &= -K^{(jk)}(x) +\mathbf{J}^{(jk)}(x),\nonumber  
\end{align}
where  
\begin{equation}
    \mathbf{J}^{(jk)}(x)=c_d \frac{x_j x_k}{|x|^{d+2}}(U(x)+1)\ind_{\{|x|\ge 1\}}.
\end{equation} 
Hence, 
\[
    \mathbf{R}^{(jk)}(f)(x)=-R^{(jk)}f(x)+\mathbf{J}^{(jk)}\ast f(x). 
\]
Exactly as in Corollary~\ref{lem:Ulimit}, we have, with dimensional  constant,  
\begin{align*}
    |\mathbf{J}^{(jk)}(x)|
    \leq \frac{C_1}{|x|^{d+1}}\ind_{\{|x|\ge C_2\}} +\frac{C_3}{|x|^d}\ind_{\{1\le |x|< C_4\}}. 
\end{align*}
Thus,  for all $j, k$, $\|\mathbf{J}^{(jk)}\|_{L^{1}(\R^d)}\leq C_d$.  It follows that for $j\ne k$,
\begin{align}\label{Contbounds0}
    \|2\mathbf{R}^{(jk)}\|_{L^p\to L^p} \leq \|2R^{(jk)}\|_{L^p\to L^p}+ C_d
     = (p^*-1)+C_d,
\end{align}
and 
\begin{align}\label{Contbounds1}
|\mathbf{R}^{(jj)}-\mathbf{R}^{(kk)}\|_{L^p\to L^p}\leq \|R^{(jj)}-R^{(kk)}\|_{L^p\to L^p}+C_d=(p^*-1)+C_d. 
\end{align} 
For $j=k$, 
\begin{align}\label{Contbounds2}
\|\mathbf{R}^{(jj)}\|_{L^p\to L^p}\leq \|R^{(jj)}\|_{L^p\to L^p}+C_d=\gamma(p)+C_d. 
\end{align}

Similarly, since  $\|\mathbf{J}^{(jk)}\|_{L^{\infty}(\R^d)}\leq C_d$, where  $C_d$ depends only on $d$, the weak-type $(1, 1)$ inequality for the classical Riesz transforms $R^{(jk)}$ gives the same for the operators $\mathbf{R}^{(jk)}$. That is, denoting the Lebesgue measure of a subset $A$ in $\R^d$ by $|A|$ we have, 
\[
    |\{x\in \R^d: |\mathbf{R}^{(jk)}(f)(x)|>\lambda\}|\leq \frac{C_d}{\lambda}\int_{\R^d} |f(x)| dx,  \quad \lambda>0, \quad f\in L^1.      
\]

Here too verifying  upper bounds for the $p$-norm of the operators $\mathbf{R}^{(jk)}$  of the form $C(p^*-1)$,  with $C$ independent of $d$, would be of interest. As before, we have the more ambitious conjecture.

\begin{conjecture}\label{ProbContSharp}
For $j\neq k$,
\begin{align} 
    &\|2\mathbf{R}^{(jk)}\|_{L^p\to L^p} =(p^*-1)\label{ProbContSharp1}, \\
    &\|\mathbf{R}^{(jj)}-\mathbf{R}^{(kk)}\|_{L^p\to L^p}= (p^*-1)\label{ProbContSharp2}.
\end{align} 
When $j=k$, 
\begin{equation}\label{ProbContSharp3}
    \|\mathbf{R}^{(jj)}\|_{L^p\to L^p}=\gamma(p).
\end{equation} 
\end{conjecture}

\begin{proof}[Proof of Theorem~\ref{thm:contiCZ}] 
We remind the reader that in what follows, $C_d$ is a dimensional constant that may change from line to line. Since 
\[
    \mathbf{K}^{(jk)}(x) = -K^{(jk)}(x) +\mathbf{J}^{(jk)}(x),
\]
with  $|\wh{K^{(jk)}}(\xi)| =\frac{|\xi_j\xi_k|}{|\xi|^2}$ and $\mathbf{J}^{(jk)}\in L^1$, it follows $\wh{\mathbf{K}^{(jk)}}\in L^{\infty}$. On the other hand, since $|U(x)+1|$ is bounded for $|x|\ge 1$ by Corollary~\ref{lem:Ulimit}, we have $|\mathbf{K}^{(jk)}(x)|\le C_d|x|^{-d}$.  It suffices to show then that $|\nabla \mathbf{K}^{(jk)}(x)|\le C_d|x|^{-d-1}$. By the definition of $\mathbf{K}^{(jk)}$, we have
\begin{align*}
    \left|\nabla\mathbf{K}^{(jk)}(x)\right|
    &= \left|\nabla K^{(jk)}(x)(-1+(U(x)+1)\ind_{\{|x|\ge 1\}})+ K^{(jk)}(x) \nabla U(x)\ind_{\{|x|\ge 1\}}\right|\\
    &\le \frac{C_d}{|x|^{d+1}}+\frac{C_d}{|x|^{d}}\left|\nabla U(x)\ind_{\{|x|\ge 1\}}\right|. 
\end{align*}

We claim that $|\nabla U(x)|\le C_d|x|^{-1}$ for $|x|\ge 1$. Let $k,l\in\{1,2,\ldots,d\}$.  Note that
\[
    \frac{\partial}{\partial x_k}U(x)
    =\frac{1}{\pi^{d/2}\Gamma((d+2)/2)}\int_0^\infty\int_{\R^d} s^{\tfrac{d}{2}}e^{-s-|y|^2}
    \frac{\partial}{\partial x_k} \left(\frac{1}{H(\tfrac{y|x|}{2\sqrt{s}}+\tfrac{x}{2},\tfrac{|x|^2}{4s})}\right) 
    \, dyds
\]
and 
\[
    \frac{\partial}{\partial x_k} 
    \left(\frac{1}{H(\tfrac{y|x|}{2\sqrt{s}}+\tfrac{x}{2},\tfrac{|x|^2}{4s})}\right) 
    =
    \sum_{l=1}^d \frac{\partial_{x_l}H_t(w)}{H_t(w)^{2}} 
    \frac{\partial w_l}{\partial x_k} 
    + \frac{\partial_{t}H_t(w)}{H_t(w)^{2}} 
    \frac{\partial t}{\partial x_k}, 
\]
where $w=\tfrac{y|x|}{2\sqrt{s}}+\tfrac{x}{2}$ and $t=\tfrac{|x|^2}{4s}$. Using $|u|e^{-u^2}\le c e^{-\tfrac{u^2}{2}}$, we have
\begin{align*}
    \left|\partial_{x_l}H_t(x)\right|
    &= \left|\sum_{m\in\Z^d}P(x-m,t)\frac{(x_l-m_l) }{t}\right|\\
    &\leq (2\pi t)^{-\tfrac{d}{2}}\sum_{m\in\Z^d}e^{-\tfrac{|x-m|^2}{2t}}\frac{|x_l-m_l| }{t}\\
    &\le \frac{C_d}{\sqrt{t}}H_{2t}(x)
\end{align*}
and similarly $\left|\partial_{t}H_{t}(x)\right| \le \frac{C}{t}H_{2t}(x)$. Thus,
\begin{align*}
    \left|\frac{\partial}{\partial x_k} 
    \left(\frac{1}{H(\tfrac{y|x|}{2\sqrt{s}}+\tfrac{x}{2},\tfrac{|x|^2}{4s})}\right) \right|
    &\le
    \sum_{l=1}^d \frac{\left|\partial_{x_l}H_t(w)\right|}{H_t(w)^{2}} \left|\frac{y_l}{2\sqrt{s}}\frac{x_k}{|x|}\right|
    + \frac{\left|\partial_{t}H_t(w)\right|}{H_t(w)^{2}} \left|\frac{x_k}{2s}\right|\\
    &\le 
    \frac{C}{|x|}\frac{H_{2t}(w)}{H_t(w)^{2}}
    \left(\sum_{l=1}^d |y_l|+1\right).
\end{align*}
If $t= \tfrac{|x|^2}{4s}\ge 1$, then it follows from the proof of Lemma~\ref{lem:hlimit} that $ \frac{H_{2t}(x)}{H_t(x)^2}\le C $ for some $C>0$, for all $x$ and $t\ge 1$.  Thus,
\begin{align*}
   & \left|\int_{\R^d}\int_0^{|x|^2 /4} s^{\tfrac{d}{2}}e^{-s-|y|^2} \frac{\partial}{\partial x_k} \left(\frac{1}{H(\tfrac{y|x|}{2\sqrt{s}}+\tfrac{x}{2},\tfrac{|x|^2}{4s})}\right)  \, dsdy\right|\\
    &\le \frac{C}{|x|}  \int_{\R^d}\int_0^{|x|^2 /4} s^{\tfrac{d}{2}}e^{-s-|y|^2}  \left(\sum_{l=1}^d |y_l|+1\right) \, dsdy\\
  &=\frac{C_d}{|x|} \left(\sum_{l=1}^d \int_{\bR^d}e^{-|y|^2}(|y_l+1)dy\right)\left(\int_0^{\frac{|x|^2}{4}}s^{\frac{d}{2}} e^{-s} ds\right)\leq 
 \frac{C_d}{|x|}.
\end{align*}
If $t=\frac{|x|^2}{4s}< 1$, we have
\[
    \frac{H(x,2t)}{H(x,t)^2}
    \le  C_d t^{\frac{d}{2}} e^{\frac{C_1}{t}}
    =C_d |z|^d s^{-\frac{d}{2}} e^{C_1\frac{ s}{|x|^2}}, 
\]
by Lemma~\ref{lem:hlimit}. Thus, 
\begin{align*}
    &\left|\int_{\R^d}\int_{|x|^2 /4}^\infty s^{\tfrac{d}{2}}e^{-s-|y|^2} \frac{\partial}{\partial x_k} \left(\frac{1}{h(\tfrac{y|x|}{2\sqrt{s}}+\tfrac{x}{2},\tfrac{|x|^2}{4s})}\right)  \, dsdy\right|\\
    &\qquad\le C|x|^{d}  \int_{\R^d}\int_{|x|^2 /4}^\infty  
    \exp\left(-|y|^2-\left(1-\frac{C_1}{|x|^2}\right)s\right)
    \left(\sum_{l=1}^d |y_l|+1\right)
    \, dsdy\\
    &\qquad\le C |x|^{d}e^{-\frac{1}{4} |x|^2}, 
\end{align*}
for $|x|\ge 2\sqrt{C_1}$. If $1\le |x|\le 2\sqrt{C_1}$, $|\nabla U(x)|$ is bounded by continuity. Therefore, we conclude that $|\nabla U(x)|\le C|x|^{-1}$,  for $|x|\ge 1$,  as desired.
\end{proof}

\subsection*{Acknowledgments} 
We are grateful to the anonymous referee for the detailed report. Their careful reading, along with numerous corrections, comments, and suggestions, greatly improved the paper. We also thank Fabrice Baudoin for valuable discussions on heat kernel estimates.
This research was conducted while the second author was a postdoctoral researcher at the Georgia Institute of Technology. He thanks the School of Mathematics at Georgia Tech for its hospitality and support during this time.

\end{document}